\documentclass[11pt, a4paper]{article}
\usepackage{float}
\usepackage{titlesec} %自定义多级标题格式的宏包
\usepackage{geometry}
\usepackage{multirow}
\usepackage{enumerate}
\usepackage{booktabs}
\usepackage{ytableau}
\usepackage{tikz}
\usepackage{graphicx}
\usepackage{caption}
\usepackage{tikz}
\usepackage{mathrsfs}
\usepackage{makecell}
\usepackage{algorithm,algorithmic}
\usepackage{lmodern}
\usepackage{amssymb}
\usepackage{amsfonts}
\usepackage{graphicx}
\usepackage{epstopdf}
\usepackage{setspace}
\usepackage[unicode={ture},colorlinks,
            linkcolor=black,
            anchorcolor=blue,
            citecolor=green]{hyperref}        
\usepackage{appendix}
\usepackage{amsmath}

\usepackage{amsthm}
\usepackage{rotating}
\usepackage{multirow}
\usepackage{colortbl}
\definecolor{darkblue}{rgb}{0.96, 0.96, 0.86}
\definecolor{darkgray}{rgb}{1.00, 0.97, 0.90}
\usepackage{makecell}
\usepackage{pifont}
\usepackage{authblk}

\usepackage{xcolor}
\usepackage[normalem]{ulem}

\allowdisplaybreaks[4]

\newtheorem{thm}{Theorem}[section]
\newtheorem{lem}[thm]{Lemma}
\newtheorem{pr}[thm]{Proposition}
\newtheorem{cor}[thm]{Corollary}

\theoremstyle{definition}

\newtheorem{re}{Remark}[section]

\newtheorem{cx}{Conjecture}[section]

\title{The second largest eigenvalue of normal Cayley graphs on symmetric groups generated by cycles}
\author{Yuxuan Li, Binzhou Xia and Sanming Zhou\\ 
School of Mathematics and Statistics, The University of Melbourne, Parkville, VIC 3010, Australia}
\begin{document}
\maketitle
\openup 0.35 \jot % Added by Sanming Zhou

\renewcommand{\thefootnote}{\empty}%{footnote}}
\footnotetext{E-mail addresses: yuxuan11@student.unimelb.edu.au (Yuxuan Li), binzhoux@unimelb.edu.au (Binzhou Xia), sanming@unimelb.edu.au (Sanming Zhou)}

\begin{abstract}
We study the normal Cayley graphs $\mathrm{Cay}(S_n, C(n,I))$ on the symmetric group $S_n$, where $I\subseteq \{2,3,\ldots,n\}$ and $C(n,I)$ is the set of all cycles in $S_n$ with length in $I$. We prove that the strictly second largest eigenvalue of $\mathrm{Cay}(S_n,C(n,I))$ can only be achieved by at most four irreducible representations of $S_n$, and we determine further the multiplicity of this eigenvalue in several special cases. As a corollary, in the case when $I$ contains neither $n-1$ nor $n$ we know exactly when $\mathrm{Cay}(S_n, C(n,I))$ has the Aldous property, namely the strictly second largest eigenvalue is attained by the standard representation of $S_n$, and we obtain that $\mathrm{Cay}(S_n, C(n,I))$ does not have the Aldous property whenever $n \in I$. As another corollary of our main results, we prove a recent conjecture on the second largest eigenvalue of $\mathrm{Cay}(S_n, C(n,\{k\}))$ where $2 \le k \le n-2$. 
\end{abstract}

\section{Introduction}
\label{sec:Introduction}

%\begin{eqnarray}
%	\tilde{\chi}_\zeta((2,1^{n-2}))&=&\frac{\chi_\zeta((2,1^{n-2}))}{d_\zeta}\\ 
%	                               &=&\frac{\sum\limits_{\substack{\mu\vdash n-2\\\mu\subset \zeta\\ \zeta\setminus \mu ~\text{is~a~border~strip}}}(-1)^{ht(\zeta\setminus \mu)}\chi_\mu((1^{n-2})) }{d_\zeta}\\
%	                               &=&\frac{\sum\limits_{\substack{\mu\vdash n-2\\\mu\subset \zeta\\ \zeta\setminus \mu ~\text{is~a~row-strip}
%}}d_\mu -\sum\limits_{\substack{\mu\vdash n-2\\\mu\subset \zeta\\ \zeta\setminus \mu ~\text{is~a~column-strip}
%}}d_\mu  }{d_\zeta} \nonumber
%\end{eqnarray}

%\begin{cx} If $\zeta>\zeta'$ in the lexicographic order, then $$\chi_\zeta((2,1^{n-2}) )=-\chi_{\zeta'}((2,1^{n-2}))>0.$$
	
%\end{cx} 

In this paper all graphs are finite, undirected and simple, and all groups are finite. Let $\Gamma=(V(\Gamma),E(\Gamma))$ be a graph with vertex set $V(\Gamma)$ and edge set $E(\Gamma)$. Then all eigenvalues of the adjacency matrix $A(\Gamma)$ of $\Gamma$ are real, and they are referred to as the \emph{eigenvalues} of $\Gamma$. We always arrange these eigenvalues in non-ascending order as $\lambda_1\geq \lambda_2\geq \cdots\geq \lambda_n$, where $n = |V(\Gamma)|$ is the order of $\Gamma$. Whenever we want to stress the dependence of the $i$-th largest eigenvalue of a graph $\Gamma$ or a real symmetric matrix $M$, we write $\lambda_i (\Gamma)$ or $\lambda_i(M)$ in place of $\lambda_i$. It is known that the largest eigenvalue of any regular graph is equal to its degree with multiplicity identical to the number of connected components of the graph. As in \cite{LXZ}, we define the \emph{strictly second largest eigenvalue} of a regular graph to be the largest eigenvalue strictly smaller than the degree of the graph.
%When the regular graph $\Gamma$ is also connected, its degree is simple and its second largest eigenvalue is strictly smaller than the degree. 

 %The {\bf degree} of vertex $v_i$ is the number of vertices in $\Gamma$ which are adjacent to $v_i$. We say that $\Gamma$ is {\bf $k$-regular} if every vertex in $V(\Gamma)$ has the same degree $k$. Cayley graphs are an important class of regular graphs in both theory and application fields. 

Let $G$ be a finite group with identity element $e$, and $S$ an inverse-closed subset of $G\setminus \{e\}$. The \emph{Cayley~graph} on $G$ with respect to $S$, denoted by $\mathrm{Cay}(G, S)$, is the $|S|$-regular graph with vertex set $G$ and edge set $\{\{g,gs\}~|~g\in G, s\in S\}$. It is readily seen that $\mathrm{Cay}(G,S)$ is connected if and only if its \emph{connection set} $S$ is a generating subset of $G$. When $S$ is closed under conjugation, the Cayley graph $\mathrm{Cay}(G,S)$ is said to be \emph{normal}. 

It is widely known that the expansion of a graph $\Gamma$ can be measured by its \emph{isoperimetric number}, which is defined as $h(\Gamma) = \min\left\{\frac{|\partial S|}{|S|}~|~\emptyset \ne S \subset V(\Gamma), |S| \le \frac{|V(\Gamma)|}{2}\right\}$, where $\partial S=\{\{x,y\}\in E(\Gamma)~|~x\in S, y\in V(\Gamma)\setminus S\}$. It is also widely known that the isoperimetric number of a connected $k$-regular graph $\Gamma$ is determined by its \emph{spectral gap} $k-\lambda_2(\Gamma)$ thanks to the following well-known inequalities (\cite{A,AM,D,M}): $\frac{k-\lambda_{2}(\Gamma)}{2} \leq h(\Gamma) \leq \sqrt{2k(k-\lambda_{2}(\Gamma))}$. Thus it is of great importance to study the second largest eigenvalue of a connected regular graph. In particular, since many important expanders are Cayley graphs, the second largest eigenvalue of Cayley graphs has been a focus of study for a long time, where, roughly speaking, an expander is a graph with small degree and large isoperimetric number. See \cite{HLW} for a survey on expander graphs with applications and \cite[Section 8]{LZ} for a collection of results on the second largest eigenvalue of Cayley graphs. The celebrated Aldous' spectral gap conjecture asserts that the second largest eigenvalue of any connected Cayley graph on the symmetric group $S_n$ with respect to a set of transpositions is achieved by the standard representation of $S_n$. This conjecture in its general form was proved in \cite{CLR} after nearly twenty years of standing. In general, a Cayley graph $\mathrm{Cay}(S_n,S)$ on $S_n$ is said to have the \emph{Aldous property} \cite{LXZ} if its strictly second largest eigenvalue is attained by the standard representation of $S_n$. In \cite{LXZ}, the authors of the present paper identified three families of normal Cayley graphs on $S_n$ with the Aldous property, one of which can be considered as a generalization of the ``normal" case of Aldous' spectral gap conjecture.

Aldous' spectral gap conjecture has inspired much interest in determining the exact value of the second largest eigenvalues of some connected Cayley graphs on symmetric or alternating groups with connection set not necessarily formed by transpositions only. For $1\leq i\leq j\leq n$, let $r_{i,j}\in S_n$ be the permutation which maps $i, i+1, \ldots, j-1, j$ to $j, j-1, \ldots, i+1, i$, respectively, and fixes all other points in $[n] = \{1, 2, \ldots, n\}$. The Cayley graphs on $S_n$ with connection sets $\{r_{1,j}~|~ 2\leq j\leq n\}$ and $\{r_{i,j}~|~1\leq i< j \leq n\}$ are called the \emph{pancake graph} $P_n$ and \emph{reversal graph} $R_n$, respectively. In \cite{C1}, Cesi determined among other things the second largest eigenvalue of $P_n$. This result was then generalized by Chung and Tobin \cite{CT} to a family of graphs which contains all pancake graphs. In the same paper, Chung and Tobin also determined the second largest eigenvalue of $R_n$ by recursively decomposing $R_n$ into $P_n$ and copies of $R_{n-1}$. This method was further used by Huang and Huang \cite{HH} to determine the second largest eigenvalues of the Cayley graphs on the alternating group $A_n$ with connection sets $\{(1~2~i), (1~i~2)~|~ 3 \leq  i \leq  n\}$, $\{(1~ i~ j), (1~ j~ i)~|~ 2\leq i< j\leq n\}$ and $\{(i~ j~ k), (i~ k~ j)~|~1\leq i<  j<k\leq n\}$, respectively. With the help of the representation theory of symmetric groups, Siemons and Zalesski \cite{SZ} determined the second largest eigenvalue of $\mathrm{Cay}(G, C(n,k))$, where $G=S_n$ or $A_n$, and $k=n$ or $n-1$, with $C(n,k)$ the set of all $k$-cycles in $S_n$. In \cite[Theorem 1.2]{MR}, Maleki and Razafimahatratra proved that if $n-k
\ge 2$ is relatively small compared to $n$ then the second largest eigenvalue of $\mathrm{Cay}(S_n, C(n,k))$ is achieved by the standard representation of $S_n$. They conjectured \cite[Conjecture 1.4]{MR} that the same result holds for any $k$ between $2$ and $n-2$, that is, for any $n \ge 4$ and $2 \leq k \leq n-2$, the strictly second largest eigenvalue of $\mathrm{Cay}(S_n, C(n,k))$ is attained by the standard representation of $S_n$, and its value is $\frac{n-k-1}{n-1}\binom{n}{k}(k-1)!$. As proved earlier, this conjecture is true for $k = 2$ \cite{PS}, $k=3$ \cite{HH} (see also \cite[Remark 4.1]{LXZ}) and $k = 4$ \cite{HH2}, and in \cite{MR} it was proved that it is also true for $k=5$.

Motivated by the researches above, we study the normal Cayley graphs $\mathrm{Cay}(S_n, C(n,I))$ in this paper, where $\emptyset \ne I\subseteq \{2,3,\ldots,n\}$ and 
\begin{equation}
\label{eq:CnI}
C(n,I)=\cup_{k\in I} C(n,k)
\end{equation}
with $C(n,k)$ the set of all $k$-cycles in $S_n$. We prove that for any $\emptyset \ne I\subseteq \{2,3,\ldots,n\}$ the strictly second largest eigenvalue of $\mathrm{Cay}(S_n, C(n,I))$ can only be achieved by at most four different irreducible representations of $S_n$, and for some special subsets $I\subseteq \{2,3,\ldots,n\}$ we obtain further the exact value of this eigenvalue together with its multiplicity (Theorems \ref{cx:I_(n-2)}, \ref{thm:n-1}, \ref{thm:n} and \ref{thm:n_n-1}). As a corollary, we give a necessary and sufficient condition for $\mathrm{Cay}(S_n, C(n,I))$ to possess the Aldous property in the case when $I$ contains neither $n$ nor $n-1$ (Corollary~\ref{cor:aldous_n-2}), and we obtain that this graph does not have the Aldous property whenever $n \in I$ (Corollaries \ref{cor:aldous_n} and \ref{cor:aldous_n-1_n}). As another corollary, we prove the abovementioned conjecture \cite[Conjecture 1.4]{MR} (Corollary \ref{cor1}) and thus solve an open problem posed by Siemons and Zalesski in \cite{SZ}. A summary of our main results can be found in Table \ref{tab:main}, where the third column shows all possible partitions of $n$  which achieve the strictly second largest eigenvalue and the last column indicates the multiplicity of this eigenvalue. 

\begin{table}[thp]
\renewcommand\arraystretch{2.0}
\centering
%\newcolumntype{g}{>{\columncolor{darkblue}}r} 
\begin{tabular}{|c|c| c | c|}
\hline
\multicolumn{2}{|c|}{ } & &  \\ 
\multicolumn{2}{|c|}{~~~~~\bf{Subfamilies}~~~~~ } & ~~~~\bf{Partitions}~~~~ & ~~~~\bf{Multiplicity}~~~~  \\
\multicolumn{2}{|c|}{ } & &  \\ \hline
\multirow{6}*{$I \cap \{n-1,n\} = \emptyset$} & $I \cap N_2 = \emptyset$ & $(n-1,1)$ and $(2,1^{n-2})$ & $2(n-1)^2$\\ \cline{2-4}
 &$I=\{2,3\}$ & $(n-1,1)$ and $(1^n)$ & $(n-1)^2+1$\\ \cline{2-4}
 &$I\ne\{2,3\}$  & \multirow{3}*{$(1^n)$} & \multirow{3}*{1} \\ 
 & $I_{\mathrm{max}} \in N_1$ & & \\ 
 & $I \cap N_2 \ne \emptyset$  & & \\ \cline{2-4}
 & $I_{\mathrm{max}} \in N_2$ & $(n-1,1)$ & $(n-1)^2$\\ \hline
\multirow{6}*{$n-1 \in I$, $n \notin I$} &$n \in N_2$ & \multirow{2}*{$(1^n)$}  & \multirow{2}*{$1$} \\
 & $I \cap N_2 \ne \emptyset$ & & \\ \cline{2-4}
 &  $n \in N_2$ & [$(n-1,1)$ and $(2,1^{n-2})$] or & \multirow{2}*{Unknown} \\  
 &  $I \cap N_2 = \emptyset$ & [$(n-3,2,1)$ and $(3,2,1^{n-3})$] & \\ \cline{2-4}
 &  \multirow{2}*{$n \in N_1$} & $(n-1,1)$,  $(n-3,2,1)$, & \multirow{2}*{Unknown}\\ 
 & &   $(2^2,1^{n-4})$ or $(2,1^{n-2})$ & \\ \hline
\multirow{5}*{$n-1 \notin I$, $n \in I$}&  $n \in N_1$ & \multirow{2}*{$(1^n)$}  & \multirow{2}*{$1$} \\
 & $I \cap N_2 \ne \emptyset$ & & \\ \cline{2-4}
  & $n\in N_1$  & \multirow{2}*{$(n-2,1^2)$ and $(3,1^{n-3})$}   & \multirow{2}*{$\frac{(n-1)^2(n-2)^2}{2}$} \\  
  & $I \cap N_2 = \emptyset$ & & \\ \cline{2-4}
 & $n\in N_2$  & $(n-2,1^2)$ or $(2,1^{n-2})$ & Unknown \\ \hline
\multirow{3}*{$\{n-1,n\}\subseteq I$} &$n \in N_2$ & $(1^n)$, $(n-2,1^2)$ or $(2,1^{n-2})$  & Unknown \\ \cline{2-4}
 &\multirow{2}*{ $n \in N_1$}  & $(1^n)$, $(n-2,1^2)$, $(3,1^{n-3})$   &\multirow{2}*{Unknown} \\ 
 & & or $(2^2,1^{n-4})$ & \\ \hline
\end{tabular}
\caption{Summary of the main results, where $\emptyset \ne I\subseteq \{2,3,\ldots,n-1,n\}$, $I_{\mathrm{max}}$ is the largest number in $I$, $N_1$ is the set of odd positive integers, and $N_2$ is the set of even positive integers}
\label{tab:main}
\end{table} 

The remainder of this paper is structured as follows. In Section 2, we give some basic definitions and present several known results that will play a key role in the proofs of our main results. Since the irreducible characters of $(n-1)$-cycles and $n$-cycles behave differently from that of cycles with length no more than $n-2$, we divide the family of graphs $\mathrm{Cay}(S_n,C(n,I))$ into four subfamilies and investigate them separately in four sections.
More explicitly, in Sections \ref{sec:proofs}--\ref{sec:n1n} we deal with the case where $I \cap \{n-1,n\} = \emptyset$, $n-1 \in I$ but $n \not \in I$, $n \in I$ but $n-1 \not \in I$, or $\{n-1,n\} \subseteq I$, respectively. As will be seen shortly, we will use tools from the representation and character theory of finite groups together with combinatorial techniques in the proofs of our main results.

% section Introduction (end)

\section{Preliminaries} % (fold)
\label{sec:preliminaries}
 
All definitions in this section can be found in \cite{Sagan}. %\cite{I,J,JK,JL,Sagan,S1,S}. 
%The notation here is consistent with that in\cite{LXZ}.
%In view of Proposition \ref{prop:normal_eigenvalue}, if $\mathrm{Cay}(G,S)$ is connected and normal, then its second largest eigenvalue is achieved by $\rho_i$ if and only if $$ \lambda_2(\mathrm{Cay}(G,S))=\sum_{s\in S} \tilde{\chi}_i(s). $$ 
In what follows we use 
$$\widehat{G}=\{\rho_1,\rho_2,\ldots,\rho_k\}$$ to denote a complete set of inequivalent (complex) irreducible matrix representations of a group $G$, with the convention that $\rho_1$ is the trivial representation. For any $\rho_i \in \widehat{G}$, the map $$\chi_i: g \mapsto \mathrm{Trace}(\rho_i(g)),\;\, g \in G$$ is the \emph{character} of $\rho_i$, and the ratio $$\tilde{\chi}_i(g):=\frac{\chi_i(g)}{\chi_i(e)}$$ is known as the \emph{normalized character} of $\rho_i$ on $g\in G$, where $\chi_i(e)$ equals the dimension $\dim\rho_i$ of $\rho_i$. Note that $\dim \rho_1 = 1$ for the trivial representation $\rho_1$.

The following proposition enables us to express the eigenvalues of any normal Cayley graph on $G$ in terms of the irreducible characters of $G$. 

\begin{pr}{\upshape \cite{PS,Z}}\label{prop:normal_eigenvalue}
Let $\{\chi_1,\chi_2,\ldots,\chi_k\}$ be a complete set of inequivalent irreducible characters of a group $G$. Then the eigenvalues of any normal Cayley graph $\mathrm{Cay}(G, S)$ on $G$ are given by 
\begin{equation*}
\mu_j=\frac{1}{\chi_j(e)}\sum_{s\in S}\chi_j(s)=\sum_{s\in S} \tilde{\chi}_j(s),\quad j=1,2,\ldots,k.
\end{equation*}
Moreover, the multiplicity of $\mu_j$ is equal to $\sum_{1\leq i\leq k,\, \mu_i=\mu_j}\chi_i(e)^2.$
\end{pr}
 
We say that the strictly second largest eigenvalue of $\mathrm{Cay}(G,S)$ is \emph{attained} or \emph{achieved} by $\rho_i \in \widehat{G}$ if
$$
\lambda_{c+1}(\mathrm{Cay}(G,S))=\sum_{s\in S} \tilde{\chi}_i(s),
$$
where $c=[G:\langle S\rangle ]$ is the index of the subgroup $\langle S\rangle$ in $G$. Note that the largest eigenvalue $|S|$ of $\mathrm{Cay}(G, S)$ has multiplicity $c$ as $\mathrm{Cay}(G, S)$
is the union of $c$ copies of the connected Cayley graph $\mathrm{Cay}(\langle S\rangle, S)$ with degree $|S|$. So $\lambda_{c+1}(\mathrm{Cay}(G,S))$ is indeed the strictly second largest eigenvalue of $\mathrm{Cay}(G, S)$. 

A \emph{partition} of a positive integer $n$ is a sequence of positive integers $\gamma=(\gamma_1,\gamma_2,\ldots,\gamma_m)$ satisfying $\gamma_1\ge \gamma_2\ge\cdots\ge\gamma_m$ and $n=\gamma_1+\gamma_2+\cdots+\gamma_m$.  We use $\gamma \vdash n$ to indicate that $\gamma$ is a partition of $n$ and let $c_i(\gamma)$ denote the number of terms in $\gamma$ which are equal to $i$. A \emph{Young~diagram} is a finite collection of boxes arranged in left-justified rows, with the row sizes weakly decreasing. The Young diagram associated to the partition $\gamma=(\gamma_1,\gamma_2,\ldots,\gamma_m)$ is the one that has $m$ rows and $\gamma_i$ boxes on the $i$-th row. Since there is a clear one-to-one correspondence between partitions and Young diagrams, we use the two terms interchangeably.

Every permutation $\sigma$ of $S_n$ has a decomposition into disjoint cycles. The \emph{cycle type} of $\sigma$ is the partition of $n$ whose parts are the lengths of the cycles in its decomposition. It is widely known that two elements of $S_n$ are conjugates if and only if they have the same cycle type. This means that the conjugacy classes of $S_n$ are characterized by the cycle types and thus correspond to the partitions of $n$. Denote by $C(S_n,\gamma)$ the conjugacy class of $S_n$ that corresponds to the partition $\gamma\vdash n$. We use $\mathrm{sgn}(\gamma)$ to denote the sign of the permutations in $C(S_n,\gamma)$. 
%For each $\gamma\vdash n$, we define $\mathrm{sgn}(\gamma) = 1$ if all permutations in $C(S_n,\gamma)$ are even and $\mathrm{sgn}(\gamma)=-1$ otherwise.

For each partition $\zeta\vdash n$, we use $\rho_\zeta$ to denote the Specht module of $S_n$ that corresponds to $\zeta$. It is well known that $\widehat{S_n}=\{\rho_\zeta~|~\zeta\vdash n\}$ is a complete list of inequivalent irreducible representations of $S_n$. The Specht modules $\rho_{(n)}$ and $\rho_{(1^n)}$ are called the \emph{trivial} and the \emph{sign representations} of $S_n$, respectively. The \emph{standard representation} of $S_n$ just refers to $\rho_{(n-1,1)}$. Let $\chi_\zeta(\cdot)$ and $\tilde{\chi}_\zeta(\cdot)$ denote the character and normalized character of $\rho_\zeta$, respectively.  We have $\chi_{(n)}(\sigma)=\tilde{\chi}_{(n)}(\sigma)= 1$ and $\chi_{(1^n)}(\sigma)=\tilde{\chi}_{(1^n)}(\sigma)=\mathrm{sgn}(\sigma)$ for any $\sigma\in S_n$. As $\chi_{\zeta}(\cdot)$ (respectively, $\tilde{\chi}_{\zeta}(\cdot)$) is a class function on $S_n$, we use $\chi_\zeta(\gamma)$ (respectively, $\tilde{\chi}_\zeta(\gamma)$) to indicate the value of $\chi_\zeta(\cdot)$ (respectively, $\tilde{\chi}_\zeta(\cdot)$) on the conjugacy class $C(S_n,\gamma)$. 

Denote by $a=(i, j)$ the box of a Young diagram $\zeta$ in the $i$-th row and $j$-th column. Then it has \emph{hook}
$$ 
H_{a}=H_{i, j}=\left\{\left(i, j^{\prime}\right)\in \zeta: j^{\prime} \geq j\right\} \cup\left\{\left(i^{\prime}, j\right)\in \zeta: i^{\prime} \geq i\right\}
$$
with corresponding \emph{hook length}
$$
h_{a}=h_{i, j}=\left|H_{i, j}\right|.
$$
To illustrate, if $\zeta=\left(4^{2}, 3^{3}, 1\right)$, then the dotted boxes in
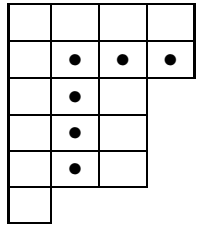
\begin{table}[H]
\centering
\begin{tabular}{|l|lll}
\hline
$\ $ & \multicolumn{1}{l|}{} & \multicolumn{1}{l|}{} & \multicolumn{1}{l|}{} \\ \hline
 & \multicolumn{1}{l|}{$\bullet$} & \multicolumn{1}{l|}{$\bullet$} & \multicolumn{1}{l|}{$\bullet$} \\ \hline
 & \multicolumn{1}{l|}{$\bullet$} & \multicolumn{1}{l|}{} &                       \\ \cline{1-3}
 & \multicolumn{1}{l|}{$\bullet$} & \multicolumn{1}{l|}{} &                       \\ \cline{1-3}
 & \multicolumn{1}{l|}{$\bullet$} & \multicolumn{1}{l|}{} &                       \\ \cline{1-3}
 &                       &                       &                       \\ \cline{1-1}
\end{tabular}
\end{table} 
\noindent  are the hook $H_{2,2}$ with hooklength $h_{2,2}=6$. The following theorem states the well-known Hook-Length Formula for the dimension $\chi_\zeta(\mathbf{1})$ of any Specht module $\rho_\zeta\in \widehat{S_n}$, where $\mathbf{1}$ is the identity element of $S_n$.

\begin{thm}{\upshape \cite[Theorem 3.10.2]{Sagan} }\label{thm:Hook} If $\zeta\vdash n$, then
	\[\chi_\zeta(\mathbf{1})=\frac{n !}{\prod_{(i, j) \in \zeta} h_{i, j}}.\]

\end{thm}

It is known \cite{F} that the character of any $\rho_\zeta\in \widehat{S_n}$ on any conjugacy class of $S_n$ is an integer with absolute value at most the dimension of $\rho_\zeta$. Hence $\tilde{\chi}_\zeta(\gamma)$ is a rational number in the interval $[-1,1]$ for all $\zeta \vdash n$ and $\gamma\vdash n$. For the convenience of the reader and in order to provide self-contained proofs, we include Table \ref{tab:tab1} from \cite{PP}, which gives the dimensions and characters of some Specht modules of $S_n$.  

\begin{table}[ht]
\centering
\begin{tabular}{c c c}
%\vspace{0.1cm}
\toprule
$~~~~\zeta\vdash n~~~$ & $~~~~\dim \rho_\zeta=\chi_\zeta(\mathbf{1})~~~~$ & $~~~~\chi_\zeta(\gamma)~\text{with}~c_i(\gamma)=c_i~~~~$\\\toprule
$(n)$ & $1$ & $1$\\ \midrule
$(n-1,1)$ & $n-1$ & $c_1-1$\\\midrule
$(n-2,2)$ & $\frac{n(n-3)}{2}$ & $\frac{c_1(c_1-3)}{2}+c_2$\\\midrule
$(n-2,1^2)$& $\frac{(n-1)(n-2)}{2}$ & $\frac{(c_1-1)(c_1-2)}{2}-c_2$\\ \midrule
$(n-3,3)$& $\frac{n(n-1)(n-5)}{6}$ & $\frac{c_1(c_1-1)(c_1-5)}{6}+(c_1-1)c_2+c_3$\\ \midrule
$(n-3,2,1)$ & $\frac{n(n-2)(n-4)}{3}$ & $\frac{c_1(c_1-2)(c_1-4)}{3}-c_3$ \\ \midrule
$(n-3,1^3)$ & $\frac{(n-1)(n-2)(n-3)}{6}$ & $\frac{(c_1-1)(c_1-2)(c_1-3)}{6}-(c_1-1)c_2+c_3$ \\ \midrule
$(n-4,2,1^2)$ & $\frac{n(n-2)(n-3)(n-5)}{8}$ & $\frac{c_1(c_1-2)(c_1-3)(c_1-5)}{8} -\frac{(c_1^2-3c_1-1)c_2}{2}-\frac{c_2^2}{2}+c_4$ \\
%$(n-4,4)$ & $\frac{n(n-1)(n-2)(n-7)}{24}$ & $\begin{aligned}
%	\frac{c_1(c_1-1)(c_1-2)(c_1-7)}{24}+\frac{(c_1^2-3c_1-1)c_2}{2}\\+\frac{c_2^2}{2}+(c_1-1)c_3+c_4~~~~~~~~~~~~~~
%\end{aligned}$ \\
\bottomrule
\end{tabular}
\vspace{0.2cm}
\caption{Dimensions and characters of some Specht modules of $\widehat{S_n}$}
\label{tab:tab1}
\end{table}

The \emph{conjugate} or \emph{transpose} of a partition $\zeta=(\zeta_1,\zeta_2,\ldots,\zeta_m)\vdash n$ is defined as $\zeta'=(\zeta_1',~\zeta_2',~\ldots,\zeta_h') \vdash n$, where $\zeta_i'$ is the length of the $i$-th column of $\zeta$. In other words, the Young diagram of $\zeta'$ is just the transpose of that of $\zeta$. % so sometimes $\zeta'$ is written as $\zeta^t$.  
The relation between $\chi_\zeta(\cdot)$ and $\chi_{\zeta'}(\cdot)$ is reflected in the following lemma.

\begin{lem}{\upshape (\cite[2.1.8]{JK}) }\label{lem:conjugate_character}
  For any $\zeta \vdash n$ and $\gamma\vdash n$, we have 
  \begin{equation*}
	\chi_{\zeta'}(\gamma)=\mathrm{sgn}(\gamma)\cdot\chi_\zeta(\gamma).
  \end{equation*} 
\end{lem}

In particular, we have $\mathrm{dim} \rho_{\zeta'}=\chi_{\zeta'}(\mathbf{1})=\chi_{\zeta}(\mathbf{1})=\mathrm{dim} \rho_\zeta$. That is, $\rho_{\zeta'}$ has the same dimension as $\rho_\zeta$. The following two lemmas give the normalized characters on $n$-cycles and $(n-1)$-cycles, respectively. 

\begin{lem}{\upshape \cite[Lemma 4.10.3]{Sagan} }\label{lem:character_for_n_cycles}
	Suppose $\zeta$ and $\gamma$ are two partitions of $n$. If $\gamma=(n)$, then 
	\[\chi_\zeta(\gamma)=\begin{cases}
		(-1)^m,& \text{if}~\zeta=(n-m,1^m)~\text{with}~0\leq m\leq n-1;\\
		0,& \text{otherwise}.
	\end{cases}\]
\end{lem}

\begin{lem}{\upshape \cite[Lemma 4.3]{SZ}} \label{lem:character_for_n-1_cycles}
	Suppose $\zeta$ and $\gamma$ are two partitions of $n$. If $\gamma=(n-1,1)$, then 
	\[\chi_\zeta(\gamma)=\begin{cases}
	    1,& \text{if}~\zeta=(n);\\
	    (-1)^{n-2},&\text{if}~\zeta=(1^n);\\
		(-1)^{m-1},& \text{if}~\zeta=(n-m,2,1^{m-2})~\text{with}~2\leq m\leq n-2;\\
		0,& \text{otherwise}.
	\end{cases}\]
\end{lem}

For any partition $\gamma$, let $\ell(\gamma)$ denote the number of parts of $\gamma$. Let $m$ and $n$ be positive integers. Given partitions $\mu$ and $\zeta$ of $m$ and $m+n$, respectively, we say that $\mu$ is a \emph{subpartition} of $\zeta$, written $\mu \subseteq \zeta$, if $\ell(\mu) \leq \ell(\zeta)$ and $\mu_{i} \leq \zeta_{i}$ for $1 \leq i \leq \ell(\mu)$. The \emph{skew diagram} $\zeta / \mu$ is defined to be the set of boxes
in $\zeta$ but not in $\mu$. Denote by $\mathrm{ht}(\zeta / \mu)$ the number of nonempty rows of $\zeta / \mu$ minus one. A skew diagram $\zeta / \mu$  is called a \emph{border strip} if it contains no subset of the $2\times 2$ box
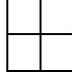
\begin{table}[H]
\centering
\begin{tabular}{|l|l|}
\hline
 &  \\ \hline
 &  \\ \hline
\end{tabular}
\end{table}
\noindent and the graph with vertices the boxes of $\zeta / \mu$ and edges joining two neighbouring boxes in the same row or column is connected. 
%We call $\gamma / \mu=(\gamma_1-\mu_1,\gamma_2-\mu_2,\ldots,\gamma_{\ell(\gamma)}-\mu_{\ell(\gamma)})$ a \emph{skew partition}, where $\mu_i=0$ if $i>\ell(\mu)$. 
For example, if $\zeta_0=(3,3,2,1)$, $\mu_1=(2,1,1)$ and $\mu_2=(3)$, then we have the skew diagrams 
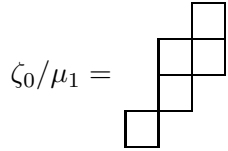
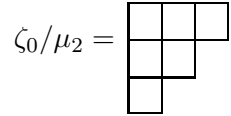
\begin{figure}[H]
\centering
\begin{minipage}[b]{0.45\linewidth}
\centering
$\zeta_0 /\mu_1 =\;\, $\begin{tabular}{lll}
\cline{3-3}
                       & \multicolumn{1}{l|}{} & \multicolumn{1}{l|}{} \\ \cline{2-3} 
\multicolumn{1}{l|}{}  & \multicolumn{1}{l|}{} & \multicolumn{1}{l|}{} \\ \cline{2-3} 
\multicolumn{1}{l|}{}  & \multicolumn{1}{l|}{} &                       \\ \cline{1-2}
\multicolumn{1}{|l|}{} &                       &                       \\ \cline{1-1}
\end{tabular}
\end{minipage}
\begin{minipage}[b]{0.45\linewidth}
\centering
$\zeta_0/\mu_2 = \;\, $\begin{tabular}{lll}                    &                       &                       \\ \hline
\multicolumn{1}{|l|}{} & \multicolumn{1}{l|}{} & \multicolumn{1}{l|}{} \\ \hline
\multicolumn{1}{|l|}{} & \multicolumn{1}{l|}{} &                       \\ \cline{1-2}
\multicolumn{1}{|l|}{} &                       &                       \\ \cline{1-1}
\end{tabular}
\end{minipage}	
\end{figure}
\noindent with $\mathrm{ht}(\zeta_0/\mu_1)=3$ and $\mathrm{ht}(\zeta_0/\mu_2)=2$. Note that these two skew diagrams are not border strips, while the following one is a border strip:
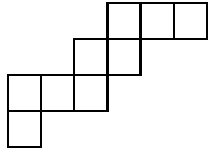
\begin{table}[H]
\centering
\begin{tabular}{llllll}
\cline{4-6}
                       &                       & \multicolumn{1}{l|}{} & \multicolumn{1}{l|}{} & \multicolumn{1}{l|}{} & \multicolumn{1}{l|}{} \\ \cline{3-6} 
                       & \multicolumn{1}{l|}{} & \multicolumn{1}{l|}{} & \multicolumn{1}{l|}{} &                       &                       \\ \cline{1-4}
\multicolumn{1}{|l|}{} & \multicolumn{1}{l|}{} & \multicolumn{1}{l|}{} &                       &                       &                       \\ \cline{1-3}
\multicolumn{1}{|l|}{} &                       &                       &                       &                       &                       \\ \cline{1-1}
\end{tabular}
\end{table} 

The main tool in this paper is as follows.

\begin{thm}[Murnaghan-Nakayama Rule]{\upshape \cite[Theorem 4.10.2]{Sagan}}\label{thm:MN}
Given positive integers $m$ and $n$, let $\rho \in S_{m+n}$ be an $m$-cycle and let $\pi$ be a permutation of the remaining $n$ elements of $[m+n]$. Then for any $\zeta\vdash m+n$,
$$
\chi_{\zeta}(\pi \rho)=\sum(-1)^{\operatorname{ht}(\zeta /\mu)} \chi_{\mu}(\pi),
$$
where the sum is over all $\mu\vdash n$ such that $\mu \subset \zeta$ and $\zeta / \mu$ is a border strip.
\end{thm}

The special case of the Murnaghan-Nakayama Rule that $\rho$ is just a $1$-cycle, that is, a fixed point, is called the \emph{Branching Rule}. To be specific, if $\zeta$ and $\gamma$ are parititions of $n+1$ with $c_1(\gamma)\ge 1$, letting $\gamma'$ be the partition of $n$ with all $c_i(\gamma')=c_i(\gamma)$ except for $c_1(\gamma')=c_1(\gamma)-1$, then
$$\chi_\zeta(\gamma)=\sum_{\zeta-\Box} \chi_{\zeta-\Box}(\gamma'),$$ where the sum is taken over all partitions of $n$ which are obtained from $\zeta$ by removing one box. 

Recall from \eqref{eq:CnI} that $C(n,I)$ is the set of all cycles in $S_n$ with lengths in $I$, where $\emptyset\ne I \subseteq \{2,3,\ldots,n\}$. Since $\mathrm{Cay}(S_n,C(n,I))$ is normal, by Proposition~\ref{prop:normal_eigenvalue} we can express its eigenvalues in terms of the irreducible characters of $S_n$. More specifically, if we denote by $\lambda_\zeta^I$ the eigenvalue of $\mathrm{Cay}(S_n,C(n,I))$ corresponding to $\zeta\vdash n$, then
\begin{eqnarray}
	\lambda_{\zeta}^I&=& \sum_{\sigma\in C(n,I)} \tilde{\chi}_\zeta(\sigma)\nonumber\\
	&=& \sum_{k\in I}|C(n,k)|\cdot \tilde{\chi}_\zeta((k,1^{n-k}))\nonumber\\
	&=&\sum_{k\in I}\binom{n}{k}(k-1)!\cdot \tilde{\chi}_\zeta((k,1^{n-k})).\label{eq:eigenvalue}
\end{eqnarray} 
Moreover, the multiplicity of $\lambda_\zeta^I$ is equal to 
\begin{equation}
\label{eq:multi}
\sum_{\substack{\mu\vdash n\\\lambda_{\mu}^I=\lambda_\zeta^I}}\chi_{\mu}(\mathbf{1})^2.
\end{equation}
In particular, by (\ref{eq:eigenvalue}), Table~\ref{tab:tab1} and Lemma~\ref{lem:conjugate_character}, for any $\emptyset \ne I \subseteq \{2,3,\ldots, n\}$, the two eigenvalues of $\mathrm{Cay}(S_n,C(n,I))$ corresponding to the sign and standard representations are
\begin{equation}
\label{eq: odd_eigenvalue_(1^n)}
\lambda_{(1^n)}^I = \sum_{k\in I} |C(n,k)|\cdot \tilde{\chi}_{(1^n)}((k,1^{n-k})) = \sum_{k\in I}\binom{n}{k}(k-1)!  \cdot (-1)^{k-1}
\end{equation}
and
\begin{equation}
\label{eq:odd_eigenvalue_(n-1,1)}
\lambda_{(n-1,1)}^I = \sum_{k\in I}  |C(n,k)|\cdot \tilde{\chi}_{(n-1,1)}((k,1^{n-k})) 
= \sum_{k\in I}   \binom{n}{k}(k-1)!\cdot \frac{n-k-1}{n-1},
\end{equation}
respectively.

\section{$\mathrm{Cay}(S_n,C(n,I))$ with $I\cap \{n-1,n\}=\emptyset$}
 % (fold)
\label{sec:proofs}

\begin{lem}\label{lem:upper_bound}
Suppose $n\ge 8$. Let $\gamma=(k,1^{n-k})$ be the cycle type of a $k$-cycle in $S_n$ with $2\leq k\leq n-3$.  For any $\zeta\vdash n$ other than $(n),(1^n),(n-1,1)$ and $(2,1^{n-2})$, we have  
\begin{equation}
\label{eq:k1n}
\tilde{\chi}_{\zeta}(\gamma)<\frac{(n-k)(n-k-1)}{n(n-1)}.
\end{equation}
\end{lem}

\begin{proof}
One can easily verify that the result holds for $n=8$. Suppose this result holds for some $n-1\ge 8$. Now we prove \eqref{eq:k1n} for $n\ge 9$ and $2\leq k\leq n-3$. Consider $k=n-3$ first. We list all the partitions of $n$ in Table~\ref{tab:tab7} which have border strips with $n-3$ boxes. The dimensions and the characters in Table~\ref{tab:tab7} are calculated with the help of the Murnaghan-Nakayama Rule and Hook-Length Formula. Also by the Murnaghan-Nakayama Rule, we know that the partitions of $n$ that are not on this list must achieve zero for the normalized character on any $(n-3)$-cycle of $S_n$. Through simple calculations with the help of Table~\ref{tab:tab7}, one can verify \eqref{eq:k1n} for $k=n-3$. 
 
\begin{table}[h]
\centering
\begin{tabular}{c c c}
%\vspace{0.1cm}
\toprule
$\zeta\vdash n$ & $\dim \rho_\zeta=\chi_\zeta(\mathbf{1})$ & $|\chi_\zeta((n-3,1^3))|$\\\toprule
$(n)$ or $(1^n)$ & $1$ & $1$\\ \midrule
$(n-1,1)$ or $(2,1^{n-2})$ & $n-1$ & $2$\\\midrule
$(n-2,1^2)$ or $(3,1^{n-3})$ & $\frac{(n-1)(n-2)}{2}$ & $1$ \\ \midrule
$(n-3,2,1)$ or $(3,2,1^{n-5})$ & $\frac{n(n-2)(n-4)}{3} $ & $1$\\ \midrule
$(n-3,3)$ or $(2^3,1^{n-6})$ & $\frac{n(n-1)(n-5)}{6}$ & $2$\\ \midrule
$(n-4,2^2)$ or $(3^2,1^{n-6})$ & $\frac{n(n-1)(n-4)(n-5)}{12}$ &  $1$ \\ \midrule
$(n-m,3,2, 1^{m-5})$  & $\frac{n!}{3(n-3)(n-m+1)(n-m-1)(m-1)(m-3)(m-5)!(n-m-3)!}$ & $2$\\\midrule
$(n-m,4,1^{m-4})$&  $\frac{n!}{6(n-3)(n-m)(n-m-1)(n-m-2)m(m-4)!(n-m-4)!}$ & $1$\\ \midrule
$(n-m,2^3,1^{m-6})$& $\frac{n!}{6(n-3)(n-m+2)(m-2)(m-3)(m-4)(m-6)!(n-m-2)!}$ & $1$\\
\bottomrule
\end{tabular}
\vspace{0.2cm}
\caption{All nonzero characters of irreducible representations of $S_n$ on $(n-3,1^3)$}
\label{tab:tab7}
\end{table}

Now suppose $2\leq k\leq n-4$. Using Table~\ref{tab:tab1} and Lemma~\ref{lem:conjugate_character}, we can show that \eqref{eq:k1n} holds for $\zeta=(n-2,2),(2^2,1^{n-4}),(n-2,1^2),(3,1^{n-3})$. In the following let $\zeta$ be any partition of $n$ with at least three boxes outside the first row and at least three boxes outside the first column. Thus, by Branching Rule, we obtain
\begin{align}
	\tilde{\chi}_{\zeta}(\gamma)&= \frac{\sum_{\zeta-\Box}\chi_{\zeta-\Box}(k,1^{n-1-k})}{\sum_{\zeta-\Box}\chi_{\zeta-\Box}(\mathbf{1})} \nonumber \\
	&\leq \max_{\zeta-\Box} \tilde{\chi}_{\zeta-\Box}((k,1^{n-1-k})) \nonumber \\
	&< \frac{(n-k-1)(n-k-2)}{(n-1)(n-2)} \nonumber \\
	&< \frac{(n-k)(n-k-1)}{n(n-1)},   \nonumber
\end{align}
where the penultimate inequality is deduced from our induction hypothesis. We can use this hypothesis because each $\zeta-\Box$ above is a partition of $n-1$ with at least two boxes outside the first row and at least two boxes outside the first column. 
\qedhere
\end{proof}

The next lemma shows that on any cycle of $S_n$ with length at most $n-2$ the normalized character of the standard representation is greater than that of those Specht modules not corresponding to $(n),(1^n),(n-1,1)$ or $(2,1^{n-2})$.  

\begin{lem}\label{lem:cycle_n-2}
	Suppose $n\ge 7$. Let $\gamma=(k,1^{n-k})$ be the cycle type of a $k$-cycle in $S_n$ with $2\leq k\leq n-2$.  For any $\zeta\vdash n$ other than $(n),(1^n),(n-1,1)$ and $(2,1^{n-2})$, we have 
            \begin{equation*} 
             \tilde{\chi}_{\zeta}(\gamma)<\tilde{\chi}_{(n-1,1)}(\gamma).
             \end{equation*}
\end{lem}
\begin{proof}
First, suppose $\gamma=(n-2,1^2)$. Table~\ref{tab:tab4} exhibits all the partitions of $n$ which have border strips with $n-2$ boxes and thus achieve nonzero characters on any $(n-2)$-cycle of $S_n$. Since $n\ge 7$, from Table~\ref{tab:tab4} one can see that $\tilde{\chi}_{(n-1,1)}(\gamma)=\frac{1}{n-1}$ and $\tilde{\chi}_\zeta(\gamma)< \frac{1}{n-1}$ whenever $\zeta\ne (n), (1^{n}),(n-1,1),(2,1^{n-2})$. 
\begin{table}[h]
\centering
\begin{tabular}{c c c c}
%\vspace{0.1cm}
\toprule
$\zeta\vdash n$ & $\dim \rho_\zeta=\chi_\zeta(\mathbf{1})$ & $\chi_\zeta((n-2,1,1))$\\\toprule %& $\chi_\zeta((n-2,2))$
$(n)$ & $1$ & $1$ \\ \midrule % & $1$
$(1^n)$ & $1$ & $(-1)^{n-3}$\\ \midrule %& $(-1)^{n-2}$
$(n-1,1)$ & $n-1$ & $1$ \\\midrule %& $-1$
$(2,1^{n-2})$ & $n-1$ & $(-1)^{n-3}$ \\\midrule%& $(-1)^{n-1}$
$(n-2,2)$ & $\frac{n(n-3)}{2}$ & $-1$ \\ \midrule %& $1 $
$(2^2,1^{n-4})$ & $\frac{n(n-3)}{2}$ & $(-1)^{n-2}$ \\ \midrule %& $(-1)^{n-2}$ 
$(n-m,3,1^{m-3})$  & $\frac{n!}{2m(n-2)(n-m)(n-m-1)(m-3)!(n-m-3)!}$ & $(-1)^{m-2}$\\\midrule% & $(-1)^{m-2}$
$(n-m,2^2,1^{m-4})$& $\frac{n!}{2(m-1)(m-2)(n-2)(n-m+1)(m-4)!(n-m-2)!}$ & $(-1)^{m-2}$\\ %& $(-1)^{m-1}$
\bottomrule
\end{tabular}
\vspace{0.2cm}
\caption{All nonzero characters of irreducible representations of $S_n$ on $(n-2,1^2)$}
\label{tab:tab4}
\end{table} 

Now suppose $\gamma=(k,1^{n-k})$ with $2\leq k\leq n-3$. One can verify that if $n=7$ then $\tilde{\chi}_{\zeta}(\gamma)<\tilde{\chi}_{(6,1)}(\gamma)$ holds for any $\zeta\ne (7), (1^7),(6,1), (2,1^5)$. If $n\ge 8$, then by Lemma~\ref{lem:upper_bound} we have 
\begin{equation*}
	\tilde{\chi}_{\zeta}(\gamma)< \frac{(n-k)(n-k-1)}{n(n-1)}<\frac{n-k-1}{n-1}=\tilde{\chi}_{(n-1,1)}(\gamma).
\end{equation*}
This completes the proof.
\end{proof}

The next two lemmas compare the eigenvalues $\lambda_{(1^n)}^I$ and $\lambda_{(n-1,1)}^I$ of $\mathrm{Cay}(S_n,C(n,I))$ for any $I\subseteq \{2,3,\ldots,n-1\}$.

\begin{lem}
\label{lem:diff_(1^n)_(n-1,1)_odd}
Suppose $n\ge 7$. If $I=\{2,3\}$, then $\lambda_{(1^n)}^I=\lambda_{(n-1,1)}^I$; if $\{2,3\}\ne I \subseteq \{2,3,\ldots,n-1\}$ and the largest number in $I$ is odd, then $\lambda_{(1^n)}^I >\lambda_{(n-1,1)}^I$. 
\end{lem}

\begin{proof}
If $I=\{2,3\}$, then a straightforward calculation using \eqref{eq: odd_eigenvalue_(1^n)} and \eqref{eq:odd_eigenvalue_(n-1,1)} yields $\lambda_{(1^n)}^I = \lambda_{(n-1,1)}^I$.

Now suppose $I\ne \{2,3\}$ and the largest number in $I$, say, $k_0$, is odd. If $k_0 = 3$, then $I = \{3\}$ and $\lambda_{(1^n)}^{\{3\}}>\lambda_{(n-1,1)}^{\{3\}}$ by \eqref{eq: odd_eigenvalue_(1^n)} and \eqref{eq:odd_eigenvalue_(n-1,1)}. It remains to consider the case where $5\leq k_0\leq n-1$. In this case, by \eqref{eq: odd_eigenvalue_(1^n)} and \eqref{eq:odd_eigenvalue_(n-1,1)}, we have
\begin{eqnarray}
& & \lambda_{(1^n)}^I-\lambda_{(n-1,1)}^I \nonumber  \\ 
&=&\binom{n}{k_0}(k_0-1)!\frac{k_0}{n-1}+\sum_{k\in I\setminus\{k_0\}}\binom{n}{k}(k-1)!\left((-1)^{k-1}-\frac{n-k-1}{n-1}\right)\nonumber\\
&\ge& \binom{n}{k_0}(k_0-1)!\frac{k_0}{n-1}-\sum\limits_{\substack{2\leq k\leq k_0-1\\ k\text{~is~even}}}\binom{n}{k}(k-1)!\frac{2n-k-2}{n-1}\nonumber\\
&=& n(n-2)!\left(\frac{1}{(n-k_0)!}-\sum\limits_{\substack{2\leq k\leq k_0-1\\ k\text{~is~even}}} \frac{2n-k-2}{k(n-k)!}\right) \nonumber\\
&=& n(n-2)!\left(\frac{1}{(n-5)!}-\sum\limits_{\substack{2\leq k\leq 5\\ k\text{~is~even}}} \frac{2n-k-2}{k(n-k)!}\right) + \nonumber\\
	 & & n(n-2)!\left(\sum\limits_{\substack{6\leq k\leq k_0-1\\ k\text{~is~even}}}\left(\frac{1}{(n-k-1)!}-\frac{1}{(n-k+1)!}-\frac{2n-k-2}{k(n-k)!}\right)\right). \label{ineq:diff_1^n_n-1,1_odd}
\end{eqnarray}
Since $n\ge 7$, we see that the first part of the lower bound \eqref{ineq:diff_1^n_n-1,1_odd} is positive. Note that when $k_0=5$ the second part of this lower bound valishes. Note also that, for $6\leq k\leq n-2$, we have
  \begin{align}
  	& \frac{1}{(n-k-1)!}-\frac{1}{(n-k+1)!}- \frac{2n-k-2}{k(n-k)!}\nonumber\\
    &>\frac{1}{(n-k-1)!}- \frac{k-2}{k(n-k)!} -\frac{2n-k-2}{k(n-k)!}\nonumber \\
    &= \frac{1}{(n-k-1)!}- \frac{2n-4}{k(n-k)!} \nonumber \\
    &\ge 0. \nonumber
  \end{align}
Therefore, the second part of the lower bound \eqref{ineq:diff_1^n_n-1,1_odd} is also positive as required to complete the proof.
\end{proof}

\begin{lem}\label{lem:diff_(1^n)_(n-1,1)_even}
Suppose $n\ge 7$. If $\emptyset \ne I\subseteq\{2,3,\ldots,n-1\}$ and the largest number in $I$ is even, then $\lambda_{(n-1,1)}^I>\lambda_{(1^n)}^I$.
\end{lem}

\begin{proof}
Denote by $k_0$ the largest number in $I$. By our assumption, $k_0$ is even. If $k_0 = 2$, then $I = \{2\}$ and $\lambda_{(n-1,1)}^{\{2\}}>\lambda_{(1^n)}^{\{2\}}$ by \eqref{eq: odd_eigenvalue_(1^n)} and \eqref{eq:odd_eigenvalue_(n-1,1)}. 
	
Now suppose $4\leq k_0\leq n-1$. By (\ref{eq: odd_eigenvalue_(1^n)}) and (\ref{eq:odd_eigenvalue_(n-1,1)}), we have
     \begin{eqnarray}
& &\lambda_{(n-1,1)}^I-\lambda_{(1^n)}^I \nonumber \\
&=& \binom{n}{k_0}(k_0-1)!\frac{2n-k_0-2}{n-1}+\sum\limits_{k\in I\setminus\{k_0\}} \binom{n}{k}(k-1)!\left(\frac{n-k-1}{n-1}+(-1)^k\right)\nonumber\\
 	&\ge& \binom{n}{k_0}(k_0-1)!\frac{2n-k_0-2}{n-1}-\sum\limits_{\substack{2\leq k\leq k_0-1\\k~\text{is~odd}}}\binom{n}{k}(k-1)!\frac{k}{n-1} \nonumber\\
 	&=& n(n-2)!\left(\frac{2n-k_0-2}{k_0(n-k_0)!}-\sum\limits_{\substack{2\leq k\leq k_0-1\\k~\text{is~odd}}}\frac{1}{(n-k)!} \right) \nonumber \\
 	&=& n(n-2)!\left(\frac{2n-6}{4(n-4)!}-\sum\limits_{\substack{2\leq k\leq 4\\k~\text{is~odd}}}\frac{1}{(n-k)!} \right) + \nonumber\\
 	& &n(n-2)!\left( \sum\limits_{\substack{5\leq k\leq k_0-1 \\k~\text{is~odd} }} \left(\frac{2n-k-3}{(k+1)(n-k-1)!}-\frac{2n-k-1}{(k-1)(n-k+1)!}-\frac{1}{(n-k)!}\right)\right).
\label{eq:9}
\end{eqnarray}
Since $n\ge 7$, the first part of the lower bound (\ref{eq:9}) is positive. Note that the second part of (\ref{eq:9}) vanishes when $k_0 = 4$. Note also that, for $5\leq k\leq n-2$, we have
\begin{align}
& \frac{2n-k-3}{(k+1)(n-k-1)!}-\frac{2n-k-1}{(k-1)(n-k+1)!}-\frac{1}{(n-k)!}\nonumber\\
	&> \frac{2n-k-3}{(k+1)(n-k-1)!}-\frac{2}{(n-k)!}\nonumber\\
	&= \frac{(2n-k-3)(n-k)-2(k+1)}{(k+1)(n-k)!}\nonumber\\
	&\ge 0. \nonumber
\end{align}
Thus, the second part of the lower bound (\ref{eq:9}) is also positive. This completes the proof. \qedhere
\end{proof}

The main result in this section is as follows.

\begin{thm}
\label{cx:I_(n-2)} 
Suppose $n\ge 7$ and $\emptyset\ne I \subseteq \{2,3,\ldots,n-2\}$. Then the following statements hold:
\begin{itemize}
	\item[\rm (a)] if $I$ only contains odd numbers, then $\mathrm{Cay}(S_n,C(n,I))$ has two connected components and its strictly second largest eigenvalue is attained by $(n-1,1)$ and $(2,1^{n-2})$ with multiplicity $2(n-1)^2$;
	\item[\rm (b)] $\mathrm{Cay}(S_n,C(n,\{2,3\}))$ is connected and its second largest eigenvalue is attained by $(n-1,1)$ and $(1^n)$ with multiplicity $(n-1)^2+1$;
	\item[\rm (c)] if $I$ contains both even and odd numbers with the largest one odd and at least $5$, then $\mathrm{Cay}(S_n,C(n,I))$ is connected and its second largest eigenvalue is attained uniquely by $(1^{n})$ with multiplicity $1$;
    \item[\rm (d)] if  the largest number in $I$ is even, then $\mathrm{Cay}(S_n,C(n,I))$ is connected and its second largest eigenvalue is attained uniquely by $(n-1,1)$ with multiplicity $(n-1)^2$.
\end{itemize}
\end{thm}

\begin{proof}
By \eqref{eq:eigenvalue} and Lemma \ref{lem:cycle_n-2}, for any $\zeta\vdash n$ other than $(n),(1^n),(n-1,1)$ and $(2,1^{n-2})$, we have
\begin{eqnarray}
	\lambda_\zeta^I &=& \sum_{k\in I} \binom{n}{k}(k-1)!\cdot \tilde{\chi}_\zeta((k,1^{n-k})) \nonumber \\
                    &<& \sum_{k\in I} \binom{n}{k}(k-1)!\cdot \tilde{\chi}_{(n-1,1)}((k,1^{n-k})) \nonumber \\
                    &=& \lambda_{(n-1,1)}^I. \nonumber
\end{eqnarray}
On the other hand, by \eqref{eq:odd_eigenvalue_(n-1,1)} we see that $\lambda_{(n-1,1)}^I$ is strictly smaller than 
$$
\lambda_{(n)}^I=\sum_{k\in I} \binom{n}{k}(k-1)!=|C(n,I)|,
$$ 
which is the degree of $\mathrm{Cay}(S_n,C(n,I))$. Therefore, the strictly second largest eigenvalue of $\mathrm{Cay}(S_n,C(n,I))$ can only be attained by partitions among $(1^n),(n-1,1)$ and $(2,1^{n-2})$. According to Lemma~\ref{lem:conjugate_character}, we have 
\begin{eqnarray}
\lambda_{(2,1^{n-2})}^I&=&\sum_{k\in I}  \binom{n}{k}(k-1)!\cdot \tilde{\chi}_{(2,1^{n-2})}((k,1^{n-k}))\nonumber\\
    	               &=& \sum_{k\in I}  \binom{n}{k}(k-1)!\cdot (-1)^{k-1} \tilde{\chi}_{(n-1,1)}((k,1^{n-k}))\nonumber\\
    	               &=&\sum_{k\in I} \binom{n}{k}(k-1)!\cdot (-1)^{k-1}\frac{n-k-1}{n-1}. \nonumber
\end{eqnarray}
Comparing with \eqref{eq:odd_eigenvalue_(n-1,1)}, we obtain that $\lambda_{(n-1,1)}^I\ge \lambda_{(2,1^{n-2})}^I$ and the strict inequality $\lambda_{(n-1,1)}^I> \lambda_{(2,1^{n-2})}^I$ holds if $I$ contains at least one even number.

In the case when $I$ contains only odd numbers, $\mathrm{Cay}(S_n,C(n,I))$ has two connected components each isomorphic to $\mathrm{Cay}(A_n, C(n,I))$ and $\lambda_{(n)}^I=\lambda_{(1^n)}^I>\lambda_{(n-1,1)}^I=\lambda_{(2,1^{n-2})}^I$. So its strictly second largest eigenvalue is only achieved by $(n-1,1)$ and $(2,1^{n-2})$. Note that both $\rho_{(n-1,1)}$ and $\rho_{(2,1^{n-2})}$ have dimension $n-1$ according to Table~\ref{tab:tab1} and Lemma~\ref{lem:conjugate_character}. We further deduce from \eqref{eq:multi} that the multiplicity of the strictly second largest eigenvalue is $2(n-1)^2$. This proves statement (a). In the other three cases there is at least one even number in $I$, and hence $\mathrm{Cay}(S_n,C(n,I))$ is connected and $\lambda_{(n-1,1)}^I>\lambda_{(2,1^{n-2})}^I$. So the second largest eigenvalue of $\mathrm{Cay}(S_n,C(n,I))$ can only be attained by $(1^n)$ or $(n-1,1)$.  Combining this with Lemmas~\ref{lem:diff_(1^n)_(n-1,1)_odd} and \ref{lem:diff_(1^n)_(n-1,1)_even}, we obtain (b), (c) and (d), where the multiplicities are calculated directly with the help of equation \eqref{eq:multi}.
\end{proof}

Theorem \ref{cx:I_(n-2)} implies the following result.  

\begin{cor}\label{cor:aldous_n-2}
Suppose $n\ge 7$ and $\emptyset\ne I \subseteq \{2,3,\ldots,n-2\}$. Then $\mathrm{Cay}(S_n,C(n,I))$ has the Aldous property if and only if one of the following conditions holds:
\begin{itemize}
	\item[\rm (a)]  $I=\{2,3\}$;
	\item[\rm (b)]  $I $ contains only odd numbers;
	\item[\rm (c)]  the largest number in $I$ is even.
\end{itemize}
\end{cor}

The next corollary of Theorem \ref{cx:I_(n-2)} confirms Conjecture 1.4 in \cite{MR}.

\begin{cor}\label{cor1}
For any $n \ge 4$ and $2 \leq k \leq n-2$, the strictly second largest eigenvalue of $\mathrm{Cay}(S_n, C(n,k))$ is attained by the standard representation of $S_n$, and its value is
\begin{eqnarray}
	\lambda_{(n-1,1)}^{\{k\}}=\frac{n-k-1}{n-1}\binom{n}{k}(k-1)!. \label{eq:cor1}
\end{eqnarray}
\end{cor}

\begin{proof}
One can easily verify this result when $n$ is $4, 5$ or $6$. Now suppose $n\ge 7$. The statements (a) and (d) in Theorem~\ref{cx:I_(n-2)} imply that the standard representation $\rho_{(n-1,1)}$ achieves the strictly second largest eigenvalue of $\mathrm{Cay}(S_n, C(n,k))$ for odd $k$ and even $k$, respectively. The value in \eqref{eq:cor1} is simply derived from \eqref{eq:odd_eigenvalue_(n-1,1)}.
\end{proof}

\section{$\mathrm{Cay}(S_n,C(n,I))$ with $I\cap\{n-1,n\}=\{n-1\}$} % (fold)
  \label{sec:n-1}

\begin{lem}\label{cor: gap_between_(1^n)_and_(n-1,1)}
Suppose $n\ge 8$ is even and $\{n-1\}\subseteq I\subseteq \{2,3,\ldots,n-1\}$. Then
$$
\lambda_{(1^n)}^I-\lambda_{(n-1,1)}^I>\frac{n(n-5)}{3}(n-3)!.
$$
\end{lem}

\begin{proof}
Since $n \ge 8$ is even, the largest number $k_0 = n-1$ in $I$ is odd and hence (\ref{ineq:diff_1^n_n-1,1_odd}) can be applied to the current situation. By this inequality, we obtain
  \begin{eqnarray*}
    \lambda_{(1^n)}^I-\lambda_{(n-1,1)}^I &\ge& n(n-2)!\left(\frac{1}{(n-5)!}-\sum\limits_{\substack{2\leq k\leq 5\\ k\text{~is~even}}} \frac{2n-k-2}{k(n-k)!}\right) + \nonumber\\
   & & n(n-2)!\left(\sum\limits_{\substack{6\leq k\leq n-2\\ k\text{~is~even}}}\frac{1}{(n-k-1)!}-\frac{1}{(n-k+1)!}-\frac{2n-k-2}{k(n-k)!}\right) \\
   &>& n(n-2)!\left(\sum\limits_{\substack{k=n-2\\ k\text{~is~even}}}\frac{1}{(n-k-1)!}-\frac{1}{(n-k+1)!}-\frac{2n-k-2}{k(n-k)!}\right) \\
   &=& \frac{n(n-5)}{3}(n-3)!
  \end{eqnarray*} 
as desired. 
\end{proof}

\begin{lem}
\label{lem:(n-m,1^m)}
Suppose $n\ge 7$. The following hold:
\begin{itemize} 
	\item[\rm (a)] if $k=n-1,n-2$ or $n-3$ and $\zeta=(n-m,1^{m})$ with $3\leq m\leq n-4$, then
    $$
    \tilde{\chi}_\zeta((k,1^{n-k}))=0;
    $$ 
	\item[\rm (b)] if $2\leq k\leq n-4$ and $\zeta=(n-m,1^{m})$ with $4\leq m\leq n-5$, then
\begin{equation}
\label{eq:3chi}
\tilde{\chi}_\zeta((k,1^{n-k}))< \tilde{\chi}_{(n-3,1^3)}((k,1^{n-k}))< \tilde{\chi}_{(n-2,1^2)}((k,1^{n-k})).
\end{equation}
\end{itemize}
\end{lem}

\begin{proof}
(a) If $k=n-1,n-2$ or $n-3$, then $\zeta=(n-m,1^{m})$ with $3\leq m\leq n-4$ does not contain any $k$-border strip. Hence $\tilde{\chi}_\zeta((k,1^{n-k}))=0$ by the Murnaghan-Nakayama Rule.

(b) Suppose $2\leq k\leq n-4$ and $\zeta=(n-m,1^{m})$, where $n \ge 7$ and $4\leq m\leq n-5$.  According to Table~\ref{tab:tab1}, we have 
\begin{eqnarray}\label{eq:n-3,1,1,1}
	\tilde{\chi}_{(n-3,1^3)}((k,1^{n-k}))=\frac{(n-k-1)(n-k-2)(n-k-3)-6(n-k-1)c_2+6c_3}{(n-1)(n-2)(n-3)}
\end{eqnarray}
\begin{eqnarray*}
	\tilde{\chi}_{(n-2,1^2)}((k,1^{n-k}))=\frac{(n-k-1)(n-k-2)-2c_2}{(n-1)(n-2)},
\end{eqnarray*}
where $c_i$ is the number of terms in $(k,1^{n-k})$ that are equal to $i$. Using these expressions, one can easily verify the second inequality in \eqref{eq:3chi}. 

It remains to prove the first inequality in \eqref{eq:3chi} for $n \ge 9$. (Note that this inequality vanishes when $n = 7$ or $8$ as $4 \le m \le n-5$ in $\zeta$.) We achieve this by induction on $n\ge 9$. Note from (\ref{eq:n-3,1,1,1}) and Lemma~\ref{lem:conjugate_character} that
\begin{eqnarray}\label{ine:m=3_m=2}
	0< \tilde{\chi}_{(n-3,1^3)}((k,1^{n-k}))=|\tilde{\chi}_{(4,1^{n-4})}((k,1^{n-k}))|.
\end{eqnarray} 
It is straightforward to verify that the first inequality in \eqref{eq:3chi} holds when $n=9$, $2\leq k\leq 5$ and $\zeta=(5,1^{4})$. Assume that $n-1\ge 9$ and for every $\zeta=(n-1-m,1^m)$ with $4\leq m\leq n-6$ and any $2\leq k\leq n-5$ the following holds:  
  \begin{eqnarray}\label{ine:induction_assumption}
  	\tilde{\chi}_\zeta((k,1^{n-1-k}))< \tilde{\chi}_{(n-4,1^3)}((k,1^{n-1-k}))=|\tilde{\chi}_{(4,1^{n-5})}((k,1^{n-1-k}))|.
  \end{eqnarray}
Now let us consider $\zeta=(n-m,1^m)\vdash n$ with $4\leq m\leq n-5$ and $2\leq k\leq n-4$. If $k=n-4$, then for any $\zeta=(n-m,1^{m})$ with $4\leq m\leq n-5$, the Young diagram of $\zeta$ contains no $(n-4)$-border strip. Thus we know from the Murnaghan-Nakayama Rule and inequality~\eqref{ine:m=3_m=2} that for $4\leq m\leq n-5,$ $$0=\tilde{\chi}_{(n-m,1^m)}((n-4,1^{4}))<\tilde{\chi}_{(n-3,1^3)}((n-4,1^{4})).$$ If $2\leq k\leq n-5$, then for any $\zeta=(n-m,1^m)$ with $4\leq m\leq n-5$, we apply the Branching Rule to the following normalized character and obtain 
\begin{eqnarray}
	& &\tilde{\chi}_\zeta((k,1^{n-k})) \nonumber\\
	%&=&\frac{\chi_{\zeta}((k,1^{n-k}))}{\chi_{\zeta}((1^n))}\nonumber\\ 
	                               &=& \frac{\sum\limits_{\zeta-\Box} \chi_{\zeta-\Box}((k,1^{n-1-k}))}{\sum\limits_{\zeta-\Box} \chi_{\zeta-\Box}((1^{n-1}))} \nonumber\\ 
	                               &=& \frac{ \chi_{(n-m,1^{m-1})}((k,1^{n-1-k})) +  \chi_{(n-m-1,1^{m})}((k,1^{n-1-k})) }{\chi_{(n-m,1^{m-1})}((1^{n-1})) +  \chi_{(n-m-1,1^{m})}((1^{n-1}))} \nonumber\\ 
	                               &\leq& \max\left\{ \frac{ \chi_{(n-m,1^{m-1})}((k,1^{n-1-k}))}{\chi_{(n-m,1^{m-1})}((1^{n-1})) } , \frac{\chi_{(n-m-1,1^{m})}((k,1^{n-1-k})) }{\chi_{(n-m-1,1^{m})}((1^{n-1}))}\right\} \nonumber\\ 
	                               %&=& \max\{ \tilde{\chi}_{(n-m,1^{m-1})}((k,1^{n-1-k})), \tilde{\chi}_{(n-m-1,1^{m})}((k,1^{n-1-k})) \} \\ \nonumber
	                               &\leq &  \tilde{\chi}_{(n-4,1^3)}((k,1^{n-1-k})) \label{ine:by_induction_assumption} \\ 
	                               &=& \frac{(n-1-k-1)(n-1-k-2)(n-1-k-3)-6(n-1-k-1)c_2+6c_3}{(n-2)(n-3)(n-4)}  \nonumber\\
	                               &<& \frac{(n-k-1)(n-k-2)(n-k-3)-6(n-k-1)c_2+6c_3}{(n-1)(n-2)(n-3)} \nonumber\\ 
	                               &=& \tilde{\chi}_{(n-3,1^3)}((k,1^{n-k})), \nonumber
\end{eqnarray}
where (\ref{ine:by_induction_assumption}) follows from the induction hypothesis~(\ref{ine:induction_assumption}).
\end{proof}

\begin{lem}
\label{re:n-m,1^m} 
Suppose $n\ge 7$ and $n\in I\subseteq \{2,3,\ldots,n\} $. Then the following hold:
\begin{itemize}
	\item[\rm (a)]  $\max_{1\leq m\leq n-1}{\lambda_{(n-m,1^m)}^I}$ can only be attained by $m=1,2,n-3,n-2$ or $n-1$; 
	\item[\rm (b)] $\lambda_{(n-2,1^2)}^I\ge\lambda_{(n-1,1)}^I$ and the equality holds if and only if $I=\{2,3,\ldots,n-2,n\}$ or $I=\{2,3,\ldots,n-1,n\}$; 
	\item[\rm (c)] if $n$ is even and $I=\{2,3,\ldots,n-2,n\}$ or $\{2,3,\ldots,n-1,n\}$, then $\lambda_{(2,1^{n-2})}^I>\lambda_{(n-1,1)}^I=\lambda_{(n-2,1^2)}^I$;
	\item[\rm (d)] $\lambda_{(n-2,1^2)}^I\ge \lambda_{(3,1^{n-3})}^I$ and the equality holds if and only if $I$ contains only odd numbers other than $n-1$ and $n-2$.
\end{itemize}
\end{lem}

\begin{proof}
    (a) Note from Table~\ref{tab:tab1} that $\tilde{\chi}_{(n-2,1^2)}((n-3,1^{3}))>0$ and $\tilde{\chi}_{(n-2,1^2)}((n-2,1^{2}))=\tilde{\chi}_{(n-2,1^2)}((n-1,1))=0$. Combining the previous lemma with Lemma~\ref{lem:character_for_n_cycles}, we obtain that for $2\leq k\leq n-3$ or $k=n$, 
	$$\tilde{\chi}_{(n-m,1^{m})}((k,1^{n-k}))<\tilde{\chi}_{(n-2,1^2)}((k,1^{n-k})), \quad 3\leq m\leq n-4$$
	and for $k=n-1$ or $n-2$, 
	$$\tilde{\chi}_{(n-m,1^{m})}((k,1^{n-k}))=0, \quad 2\leq m\leq n-3.$$
Thus equation~(\ref{eq:eigenvalue}) implies that whenever $n\in I\subseteq \{2,3,\ldots,n\}$ we have $\lambda_{(n-m,1^m)}^I<\lambda_{(n-2,1^2)}^I$ for every $3\leq m\leq n-4$, and so the maximum of $\lambda_{(n-m,1^m)}^I$ for $1\leq m\leq n-1$  can only be attained by $m=1,2,n-3,n-2$ or $n-1$.
    %$$\max_{1\leq m\leq n-1}{\lambda_{(n-m,1^m)}^I}=\max_{1\leq m\leq n-1}\sum_{k\in I} \binom{n}{k}(k-1)!\cdot \tilde{\chi}_{(n-m,1^{m})}((k,1^{n-k})) $$   
    
    (b) We have
	\begin{eqnarray}
	& & \lambda_{(n-2,1^2)}^I-\lambda_{(n-1,1)}^I \nonumber \\
	&=& \sum_{k\in I\setminus\{n\}} \binom{n}{k}(k-1)!\left(\frac{(n-k-1)(n-k-2)-2c_2}{(n-1)(n-2)}-\frac{n-k-1}{n-1}\right) + 2(n-3)!+(n-2)! \nonumber \\
	&\ge& -\sum_{2\leq k\leq n-2} \binom{n}{k}(k-1)! \frac{k(n-k-1)+2c_2}{(n-1)(n-2)}+2(n-3)!+(n-2)!\nonumber \\
	&=& n(n-3)!-\sum_{2\leq k\leq n-2} \frac{n(n-3)!}{(n-k)(n-k-2)!}-\frac{n}{n-2} \nonumber \\
	&=& n(n-3)!\left(1-\sum_{2\leq k\leq n-2} \frac{1}{(n-k)(n-k-2)!}\right) -\frac{n}{n-2} \nonumber\\
	&=& n(n-3)!\cdot \frac{1}{(n-2)!} -\frac{n}{n-2} \nonumber \\
	&=& 0. \nonumber
\end{eqnarray}
Thus $\lambda_{(n-2,1^2)}^I\ge\lambda_{(n-1,1)}^I$ and the equality holds if and only if $I=\{2,3,\ldots,n-2,n\}$ or $I=\{2,3,\ldots,n-1,n\}$. 

(c) We have $\lambda_{(n-1,1)}^I=\lambda_{(n-2,1^2)}^I$ from (b). We claim that the partition $(2,1^{n-2})$ yields a larger eigenvalue than $(n-2,1^2)$ and $(n-1,1)$ in this case. In fact, 
\begin{eqnarray} 
	& &\lambda_{(2,1^{n-2})}^I-\lambda_{(n-1,1)}^I \nonumber\\
	&=& \sum_{k=2}^{n-2} \binom{n}{k}(k-1)!(-1)^{k-1}\frac{n-k-1}{n-1}+(n-2)! \nonumber \\
	& & -\left(\sum_{k=2}^{n-2} \binom{n}{k}(k-1)!\frac{n-k-1}{n-1}-(n-2)!\right) \nonumber \\
	&=&-2 \sum_{\substack{2\leq k \leq n-2\\ k~\mathrm{is~even}}}\binom{n}{k}(k-1)!\frac{n-k-1}{n-1} +2(n-2)!\nonumber \\
	&=& 2(n-2)!\left(1-\sum_{\substack{2\leq k \leq n-2\\ k~\mathrm{is~even}}} \frac{n}{k(n-k)(n-k-2)!}\right) \nonumber \\
	&>& 0. \nonumber 
\end{eqnarray}
The last inequality above is deduced from the fact that $n$ is at least $8$ and 
\begin{eqnarray}
	\sum_{\substack{2\leq k \leq n-2\\ k~\mathrm{is~even}}} \frac{n}{k(n-k)(n-k-2)!}&=& \frac{n}{2(n-2)}+\frac{n}{8(n-4)}+\sum_{\substack{2\leq k \leq n-6\\ k~\mathrm{is~even}}} \frac{n}{k(n-k)(n-k-2)!} \nonumber \\
	&<& \frac{n}{2(n-2)}+\frac{n}{8(n-4)}+\sum_{\substack{2\leq k \leq n-6\\ k~\mathrm{is~even}}} \frac{1}{(n-k-2)!} \nonumber \\
	&<& \frac{8}{12}+\frac{8}{32}+\frac{2}{4!} \nonumber \\
	&=& 1. \nonumber
\end{eqnarray}

(d) Note from Table~\ref{tab:tab1} that $\tilde{\chi}_{(n-2,1^2)}((k,1^{n-k}))\ge 0$ for any $2\leq k\leq n$ and that
$\tilde{\chi}_{(n-2,1^2)}((k,1^{n-k}))= 0$ if and only if $k=n-1$ or $n-2$. Thus we have $\tilde{\chi}_{(n-2,1^2)}((k,1^{n-k}))\ge \tilde{\chi}_{(3,1^{n-3})}((k,1^{n-k}))=(-1)^{k-1}\tilde{\chi}_{(n-2,1^2)}((k,1^{n-k}))$ for any $2\leq k\leq n$. This implies that $\lambda_{(n-2,1^2)}^I\ge \lambda_{(3,1^{n-3})}^I$ and the equality holds if and only if $I$ contains only odd numbers other than $n-1$ and $n-2$.  
\end{proof}

\begin{lem}\label{lem:(n-m,2,1^m-2)_1}
Suppose $n\ge 7$. If $2\leq k\leq n-5$ and $\zeta=(n-m,2,1^{m-2})$ with $4\leq m\leq n-4$, then 
\begin{equation}\label{ineq: n-m,2,1}
	\tilde{\chi}_{\zeta}((k,1^{n-k}))< \tilde{\chi}_{(n-3,2,1)}((k,1^{n-k}))<\tilde{\chi}_{(n-2,2)}((k,1^{n-k})).
\end{equation}
\end{lem}

\begin{proof}
By Table~\ref{tab:tab1}, we have 
$$
    	\tilde{\chi}_{(n-3,2,1)}((k,1^{n-k}))=\frac{(n-k)(n-k-2)(n-k-4)-3c_3}{n(n-2)(n-4)},
$$
$$
    	\tilde{\chi}_{(n-2,2)}((k,1^{n-k}))=\frac{(n-k)(n-k-3)+2c_2}{n(n-3)},
$$
where $c_i$ is the number of terms in $(k,1^{n-k})$ which are equal to $i$. Using these expressions and Lemma~\ref{lem:conjugate_character}, one can easily verify that 
$$
|\tilde{\chi}_{(3,2,1^{n-5})}((k,1^{n-k}))|=\tilde{\chi}_{(n-3,2,1)}((k,1^{n-k}))<\tilde{\chi}_{(n-2,2)}((k,1^{n-k}))
$$ 
for $n\ge 7$ and $2\leq k\leq n-5$.  Note that the first inequality in \eqref{ineq: n-m,2,1} vanishes when $n = 7$ as we require $4 \le m \le n-4$ in $\zeta$.  
In the following we prove by induction on $n\ge 8$ that for $2\leq k\leq n-5$, 
$$
\tilde{\chi}_{(n-m,2,1^{m-2})}((k,1^{n-k}))<\tilde{\chi}_{(n-3,2,1)}((k,1^{n-k}))
$$ 
holds for $4\leq m\leq n-4$. One can check that this holds for $n=8$ and $9$. Suppose the above inequality holds for some $n-1\ge 9$, that is, for $2\leq k\leq n-6$,
     \begin{eqnarray}
     \label{eq:(n-m,2,1^m-2)_assumption}
     	\tilde{\chi}_{(n-1-m,2,1^{m-2})}((k,1^{n-1-k}))<\tilde{\chi}_{(n-4,2,1)}((k,1^{n-1-k}))=|\tilde{\chi}_{(3,2,1^{n-6})}((k,1^{n-1-k}))|
     \end{eqnarray} 
holds for $4\leq m\leq n-5$. 

Now we consider $\tilde{\chi}_{(n-m,2,1^{m-2})}((k,1^{n-k}))$ with $4\leq m\leq n-4$ and $2\leq k\leq n-5$. Note that $\tilde{\chi}_{(n-3,2,1)}((k,1^{n-k}))>0$ for $n\ge 8$ and $2\leq k\leq n-5$. First, let $k=n-5$. By the Murnaghan-Nakayama Rule we see that if $m\ne 5, n-5$, then $\tilde{\chi}_{\zeta}((n-5,1^{5}))=0<\tilde{\chi}_{(n-3,2,1)}((n-5,1^{5}))$ as there is no $(n-5)$-border strip in $\zeta$. If $m=5$ or $n-5$, then by a simple computation we still have 
$$
0<|\tilde{\chi}_{(n-m,2,1^{m-2})}((n-5,1^{5}))|<\tilde{\chi}_{(n-3,2,1)}((n-5,1^{5})).
$$ 
Thus, for $4\leq m\leq n-4$,
	 \begin{eqnarray*}
	 	\tilde{\chi}_{(n-m,2,1^{m-2})}((n-5,1^{5}))<\tilde{\chi}_{(n-3,2,1)}((n-5,1^{5})).
	 \end{eqnarray*}  
Next, let $2\leq k\leq n-6$. For every $\zeta=(n-m,2,1^{m-2})$ with $4\leq m\leq n-4$, we have
      \begin{eqnarray}
		& &\tilde{\chi}_{\zeta}((k,1^{n-k}))\nonumber \\ \nonumber
		%&=&\frac{\chi_{\zeta}((k,1^{n-k}))}{\chi_{\zeta}((1^n))}\\ \nonumber
		&=&\frac{\sum\limits_{\zeta-\Box} \chi_{\zeta-\Box}((k,1^{n-1-k}))}{\sum\limits_{\zeta-\Box} \chi_{\zeta-\Box}((1^{n-1}))} \\ \nonumber
	    &=& \frac{ \chi_{(n-m-1,2,1^{m-2})}((k,1^{n-1-k})) +  \chi_{(n-m,2,1^{m-3})}((k,1^{n-1-k})) +  \chi_{(n-m,1^{m-1})}((k,1^{n-1-k}))}{\chi_{(n-m-1,2,1^{m-2})}((1^{n-1})) +  \chi_{(n-m,2,1^{m-3})}((1^{n-1})) + \chi_{(n-m,1^{m-1})}((1^{n-1}))} \\ \nonumber
	    &\leq & \max\left\{ \tilde{\chi}_{(n-m-1,2,1^{m-2})}((k,1^{n-1-k})),  \tilde{\chi}_{(n-m,2,1^{m-3})}((k,1^{n-1-k})),  \tilde{\chi}_{(n-m,1^{m-1})}((k,1^{n-1-k})\right\} \\ \label{ineq:induction_assumption}
	    &\leq & \max\left\{ \tilde{\chi}_{(n-4,2,1)}((k,1^{n-1-k})), \tilde{\chi}_{(n-4,1^3)}((k,1^{n-1-k}) \right\} \\ \nonumber
	    &=& \max\Bigg\{ \frac{(n-1-k)(n-1-k-2)(n-1-k-4)-3c_3}{(n-1)(n-3)(n-5)}, \\ \nonumber
	    & & \frac{(n-k-2)(n-k-3)(n-k-4)-6(n-k-2)c_2+6c_3}{(n-2)(n-3)(n-4)}\Bigg\} \\ \nonumber
	    &<& \frac{(n-k)(n-k-2)(n-k-4)-3c_3}{n(n-2)(n-4)} \\ \nonumber
	    &=& \tilde{\chi}_{(n-3,2,1)}((k,1^{n-k})),
	\end{eqnarray}
	where (\ref{ineq:induction_assumption}) is deduced from the induction hypothesis (\ref{eq:(n-m,2,1^m-2)_assumption}) as well as Lemma~\ref{lem:(n-m,1^m)}.
\end{proof}

\begin{lem}
\label{lem:(n-m,2,1^m-2)_2}
Suppose $n\ge 7$. Then for any $\zeta=(n-m,2,1^{m-2})$ with $4\leq m\leq n-4$ the following hold:
\begin{itemize}
\item[\rm (a)] $\tilde{\chi}_{\zeta}((n-1,1))< \tilde{\chi}_{(n-3,2,1)}((n-1,1))< |\tilde{\chi}_{(n-2,2)}((n-1,1))|$;
\item[\rm (b)] $\sum_{k\in I } \binom{n}{k}(k-1)! \tilde{\chi}_{\zeta}((k,1^{n-k}))<\sum_{k\in I } \binom{n}{k}(k-1)! \tilde{\chi}_{(n-3,2,1)}((k,1^{n-k}))$ for every $I$ with $n-1\in I\subseteq\{n-4,n-3,n-1\}$.
	  \end{itemize} 	 
\end{lem}

\begin{proof}
(a) We can obtain the following facts by using Lemma~\ref{lem:character_for_n-1_cycles} and the Hook-Length Formula directly: If $n=7$, then $\tilde{\chi}_{(n-3,2,1)}((n-1,1))<|\tilde{\chi}_{(n-2,2)}((n-1,1))|$;
if $n=8$ or $9$, then $|\tilde{\chi}_{(n-4,2,1^2)}((n-1,1))|< \tilde{\chi}_{(n-3,2,1)}((n-1,1))< |\tilde{\chi}_{(n-2,2)}((n-1,1))|$; if $n\ge 10$, then for any $5\leq m\leq n-5$,
	\begin{eqnarray}
		\tilde{\chi}_{(n-m,2,1^{m-2})}((n-1,1))&\leq& \tilde{\chi}_{(n-5,2,1^3)}((n-1,1)) \nonumber \\
		&<&|\tilde{\chi}_{(n-4,2,1^2)}((n-1,1))|\nonumber \\
		&<& \tilde{\chi}_{(n-3,2,1)}((n-1,1)) \nonumber \\
		&<& |\tilde{\chi}_{(n-2,2)}((n-1,1))|. \nonumber
	\end{eqnarray}

(b) One can check that the inequality holds for $n = 8, 9$. Now suppose $n\ge 10$. If $k=n-4$, then for $\zeta=(n-m,2,1^{m-2})$ with $3\leq m\leq n-3$, we have $\tilde{\chi}_{\zeta}((n-4,1^4))\ne 0$ if and only if $m=4$ or $n-4$. Using Table~\ref{tab:tab1}, we obtain that $$(-1)^{n-5}\tilde{\chi}_{(4,2,1^{n-6}))}((n-4,1^{4}))=\tilde{\chi}_{(n-4,2,1^2))}((n-4,1^{4}))< 0=\tilde{\chi}_{(n-3,2,1))}((n-4,1^{4})).$$ If $k=n-3$, then $\tilde{\chi}_\zeta((n-3,1^3))=0$ for $\zeta=(n-m,2,1^{m-2})$ with $4\leq m\leq n-4$ and $$(-1)^{n-4}\cdot\tilde{\chi}_{(3,2,1^{n-5})}((n-3,1^3))=\tilde{\chi}_{(n-3,2,1)}((n-3,1^3))=\frac{-3}{n(n-2)(n-4)}.$$ Combining these facts with the first part of this lemma, we have for every $\zeta=(n-m,2,1^{m-2})$ with $5\leq m\leq n-5$ and every $I$ with $n-1\in I\subseteq \{n-4,n-3,n-1\}$, 
$$
\sum_{k\in I } \binom{n}{k}(k-1)! \tilde{\chi}_{\zeta}((k,1^{n-k}))\leq\sum_{k\in I } \binom{n}{k}(k-1)! \tilde{\chi}_{(n-5,2,1^3)}((k,1^{n-k})).
$$
To complete the proof, it remains to establish the required inequality for $\zeta=(n-m,2,1^{m-2})$ with $m=4,5$ or $n-4$ and $n-1\in I\subseteq \{n-4,n-3,n-1\}$. This can be done by straightforward computations with the help of Lemma~\ref{lem:character_for_n-1_cycles} and Table~\ref{tab:tab1}. 
\end{proof}

\begin{re}
\label{re:(n-m,2,1^m-2)}
Note from Table~\ref{tab:tab4} that $\tilde{\chi}_{(n-m,2,1^{m-2})}((n-2,1^2))=0$ for $3\leq m\leq n-3$ and from Lemma~\ref{lem:character_for_n_cycles} that $\tilde{\chi}_{(n-m,2,1^{m-2})}((n))=0$ for $2\leq m\leq n-2$. Combining these with Lemmas~\ref{lem:(n-m,2,1^m-2)_1} and \ref{lem:(n-m,2,1^m-2)_2}, we obtain that for any $\zeta=(n-m,2,1^{m-2})$ with $4\leq m\leq n-4$ and every $I$ with $n-1\in I \subseteq \{2,3,\ldots,n-1,n\}$,
\begin{eqnarray}
	\lambda_\zeta^I&=& \sum_{k\in I} \binom{n}{k}(k-1)! \cdot \tilde{\chi}_\zeta((k,1^{n-k})) \nonumber \\
                   &<& \sum_{k\in I} \binom{n}{k}(k-1)! \cdot \tilde{\chi}_{(n-3,2,1)}((k,1^{n-k}))\nonumber \\
                   &=& \lambda_{(n-3,2,1)}^I. \nonumber
\end{eqnarray}
Thus the maximum of $\lambda^I_{(n-m,2,1^{m-2})}$ for $2\leq m\leq n-2$ can only be attained by $m = 2,3,n-3,n-2$.
\end{re}

Now we are ready to prove our main result in this section.

\begin{thm}\label{thm:n-1}
Suppose $n\ge 7$ and $\{n-1\}\subseteq I \subseteq \{2,3,\ldots,n-1\}$. Then the following statements hold:
\begin{itemize}
\item[\rm (a)] if $n$ is even and $I$ contains at least one even number, then the second largest eigenvalue of $\mathrm{Cay}(S_n,C(n,I))$ is attained uniquely by $(1^n)$, and moreover the multiplicity of this eigenvalue is $1$;  
%\item[1.] If $I$ contains only odd numbers, then the second largest eigenvalue $\mathrm{Cay}(S_n,C(n,I))$ is attained exactly by $(n-1,1)$ and $(2,1^{n-2})$.
\item[\rm (b)] if $n$ is even and $I$ contains only odd numbers, then the strictly second largest eigenvalue of $\mathrm{Cay}(S_n,C(n,I))$ can only be attained by either $(n-1,1)$ and $(2,1^{n-2})$ or $(n-3,2,1)$ and $(3,2,1^{n-5})$;  
\item[\rm (c)] if $n$ is odd, then the second largest eigenvalue of $\mathrm{Cay}(S_n,C(n,I))$ is attained by $(n-1,1)$, $(n-3,2,1)$, $(2^2,1^{n-4}))$ or $(2,1^{n-2})$. %Furthermore, if $n-3$ or $n-2$ belongs to $I$,then the second largest eigenvalue of $\mathrm{Cay}(S_n,C(n,I))$ can only be attained by $(n-1,1)$ or $(2,1^{n-2})$.
\end{itemize}
\end{thm}

\begin{proof}
(a) For any $\zeta\vdash n$, we have  
\begin{eqnarray*}
\lambda_\zeta^I&=& \sum_{k\in I} \binom{n}{k}(k-1)! \tilde{\chi}_\zeta((k,1^{n-k})) \\
&=& \sum_{\substack{k\in I\setminus\{n-1\}}} \binom{n}{k}(k-1)! \tilde{\chi}_\zeta((k,1^{n-k}))  + n(n-2)!\tilde{\chi}_\zeta((n-1,1)).
\end{eqnarray*}
According to Lemma~\ref{lem:character_for_n-1_cycles}, the second term above vanishes unless $\zeta = (n),(1^n)$ or $(n-m,2,1^{m-2})$ with $2\leq m\leq n-2$. This together with Lemma~\ref{lem:cycle_n-2} implies that, for $\zeta\ne (n),(1^n),(n-1,1),(2,1^{n-2})$, $(n-m,2,1^{m-2})$ with $2\leq m\leq n-2$, we have
\begin{eqnarray}
\lambda_\zeta^I&=&\sum_{\substack{k\in I\setminus\{n-1\}}} \binom{n}{k}(k-1)! \tilde{\chi}_\zeta((k,1^{n-k})) \nonumber\\
&<& \sum_{\substack{k\in I\setminus\{n-1\}}} \binom{n}{k}(k-1)! \tilde{\chi}_{(n-1,1)}((k,1^{n-k})) \nonumber\\
&=& \lambda_{(n-1,1)}^I. \label{ineq:thm_n-1-a}
\end{eqnarray}
Since $n$ is even and the largest number $n-1$ in $I$ is odd, Lemma~\ref{lem:diff_(1^n)_(n-1,1)_odd} implies that $\lambda_{(1^n)}^I>\lambda_{(n-1,1)}^I$. Moreover, as $I$ contains at least one even number smaller than $n-1$, we also have 
\begin{eqnarray}
	\lambda_{(n-1,1)}^I &=& \sum_{k\in I} \binom{n}{k}(k-1)!\cdot \frac{n-k-1}{n-1} \nonumber\\
                        &>& \sum_{k\in I} \binom{n}{k}(k-1)!\cdot (-1)^{k-1}\frac{n-k-1}{n-1} \nonumber\\
                        &=& \lambda_{(2,1^{n-2})}^I. \label{eq:25}
\end{eqnarray}
So the second largest eigenvalue of $\mathrm{Cay}(S_n,C(n,I))$ can only be attained by $(1^n)$ or $(n-m,2,1^{m-2})$ with $2\leq m\leq n-2$.

In the following we assume $\zeta=(n-m,2,1^{m-2})$ with $2\leq m\leq n-2$, and we aim to show that $\zeta$ does not give the second largest eigenvalue of $\mathrm{Cay}(S_n,C(n,I))$. In fact, by Lemma~\ref{lem:cycle_n-2},
\begin{eqnarray}
\lambda_\zeta^I &=& \sum_{\substack{k\in I\setminus\{n-1\}}} \binom{n}{k}(k-1)! \tilde{\chi}_\zeta((k,1^{n-k}))  + n(n-2)!\tilde{\chi}_\zeta((n-1,1))\nonumber \\
&<& \sum_{\substack{k\in I\setminus\{n-1\}}} \binom{n}{k}(k-1)! \tilde{\chi}_{(n-1,1)}((k,1^{n-k}))  + n(n-2)!\tilde{\chi}_\zeta((n-1,1)) \nonumber \\
&=& \lambda_{(n-1,1)}^I + n(n-2)!\tilde{\chi}_\zeta((n-1,1)). \nonumber
\end{eqnarray}
On the other hand, by Lemma~\ref{lem:character_for_n-1_cycles},
\begin{eqnarray*}
n(n-2)!\tilde{\chi}_\zeta((n-1,1))&=& n(n-2)! \frac{\chi_\zeta((n-1,1))}{\chi_\zeta((1^n))} \\
& = & n(n-2)! \frac{(-1)^{m-1}(n-1)(n-m)(n-m-2)!m(m-2)!}{n!} \\
& \leq & n(n-2)! \frac{3(n-1)(n-3)(n-5)!}{n!} \\
& = & 3(n-3)(n-5)!,  
\end{eqnarray*}
where the second last step follows from the fact that $n$ is even and thus the maximum 
$$
\max_{2\leq m\leq n-2} \frac{(-1)^{m-1}(n-1)(n-m)(n-m-2)!m(m-2)!}{n!}
$$ 
is achieved by $m \in \{3, n-3\}$. Combining these with Lemma~\ref{cor: gap_between_(1^n)_and_(n-1,1)}, we obtain
\begin{eqnarray*}
\lambda_{(1^n)}^I & > & \lambda_{(n-1,1)}^I + \frac{n(n-5)}{3}(n-3)! \\
& > & \lambda_{(n-1,1)}^I + 3(n-3)(n-5)! \\
& \ge & \lambda_{(n-1,1)}^I + n(n-2)!\tilde{\chi}_\zeta((n-1,1)) \\
& > & \lambda_{(n-m,2,1^{m-2})}^I
\end{eqnarray*} 
for $2\leq m\leq n-2$. Therefore, the second largest eigenvalue of $\mathrm{Cay}(S_n,C(n,I))$ is attained uniquely by $(1^n)$. Moreover, the multiplicity of this eigenvalue is equal to the square of the dimension of $\rho_{(1^n)}$, namely $1$.
 
(b) Since $I$ contains only odd numbers, by Lemma~\ref{lem:conjugate_character} we have $\lambda_\zeta^I=\lambda_{\zeta'}^I$ for any $\zeta \vdash n$. According to \eqref{ineq:thm_n-1-a}, the strictly second largest eigenvalue can only be attained by $(n-1,1),(2,1^{n-2})$ or $(n-m,2,1^{m-2})$ with $2\leq m\leq n-2$. By a direct computation using Table~\ref{tab:tab1}, one can verify that for $2\leq k\leq n-1$ and $m=2$ or $n-2$, we have $\tilde{\chi}_{(n-m,2,1^{m-2})}((k,1^{n-k}))<\tilde{\chi}_{(n-1,1)}((k,1^{n-k}))$, which implies $\lambda_{(n-m,2,1^{m-2})}^I<\lambda_{(n-1,1)}^I=\lambda_{(2,1^{n-2})}^I$ when $m = 2$ or $n-2$. On the other hand, by Remark~\ref{re:(n-m,2,1^m-2)}, we have $\lambda_{(n-m,2,1^{m-2})}^I<\lambda_{(n-3,2,1)}^I=\lambda_{(3,2,1^{n-5})}^I$ for $4\leq m\leq n-4$. The result follows from these inequalities. 
 
(c) Similarly to the proof of (a) above, for $\zeta\ne (n),(1^n),(n-1,1),(2,1^{n-2})$, $(n-m,2,1^{m-2})$ with $2\leq m\leq n-2$, we have $\lambda_\zeta^I < \lambda_{(n-1,1)}^I$. Lemma~\ref{lem:diff_(1^n)_(n-1,1)_even} implies $\lambda_{(n-1,1)}^I>\lambda_{(1^n)}^I$ when $n\ge 7$ is odd and $n-1$ is the largest number in $I$. Thus the second largest eigenvalue of $\mathrm{Cay}(S_n,C(n,I))$ can only be attained by $(n-1,1),(2,1^{n-2})$ or $(n-m,2,1^{m-2})$ with $2\leq m\leq n-2$.

According to Remark~\ref{re:(n-m,2,1^m-2)}, the maximum of $\lambda_{(n-m,2,1^{m-2})}^I$ for $2\leq m\leq n-2$ is attained by $m=2,3,n-3$ or $n-2$. Furthermore, when $m=2, n-3$, Lemmas~\ref{lem:character_for_n-1_cycles} and \ref{lem:cycle_n-2} imply $\tilde{\chi}_{(n-m,2,1^{m-2})}((n-1,1))<0=\tilde{\chi}_{(n-1,1)}((n-1,1))$ and $\tilde{\chi}_{(n-m,2,1^{m-2})}((k,1^{n-k}))<\tilde{\chi}_{(n-1,1)}((k,1^{n-k}))$ for $2\leq k\leq n-2$, respectively. Therefore, for $m=2, n-3$, we have $\lambda_{(n-m,2,1^{m-2})}^I<\lambda_{(n-1,1)}^I$.
Thus the second largest eigenvalue of $\mathrm{Cay}(S_n,C(n,I))$ can only be attained by $(n-1,1),(2,1^{n-2})$ or $(n-m,2,1^{m-2})$ with $m=3,n-2$, as desired.
\end{proof}

\begin{re} 
One can verify that the strictly second largest eigenvalues of $\mathrm{Cay}(S_8,C(8,7))$ and $\mathrm{Cay}(S_8,C(8,\{3,7\}))$ are attained by $(5,2,1)$ and $(3,2,1^3)$, while the strictly second largest eigenvalues of $\mathrm{Cay}(S_8,C(8,\{5,7\}))$ and $\mathrm{Cay}(S_8,C(8,\{3,5,7\}))$ are attained by $(7,1)$ and $(2,1^6)$. This shows that both cases in part (b) of Theorem \ref{thm:n-1} can occur. However, we do not know whether the four partitions in part (b) of Theorem \ref{thm:n-1} can achieve the strictly second largest eigenvalue simultaneously. 
\end{re}

Part (a) of Theorem \ref{thm:n-1} implies the following result.  

\begin{cor}\label{cor:aldous_n-1}
Suppose $n\ge 8$ is even and $\{n-1\}\subseteq I \subseteq \{2,3,\ldots,n-1\}$. If $I$ contains at least one even number, then $\mathrm{Cay}(S_n, C(n,I))$ does not possess the Aldous property. 
\end{cor}

More work is required to determine when $\mathrm{Cay}(S_n, C(n,I))$ has the Aldous property under the conditions of parts (b) and (c) of Theorem \ref{thm:n-1}.  

In \eqref{eq:25} we saw that $\lambda_{(n-1,1)}^I > \lambda_{(2,1^{n-2})}^I$ whenever $I\subseteq \{2,3\ldots,n-1\}$ contains at least one even number smaller than $n-1$. In general, by Lemma~\ref{lem:conjugate_character} and \eqref{eq:odd_eigenvalue_(n-1,1)}, for any $I\subseteq \{2,3\ldots,n-1\}$ we have
\begin{align}
	\lambda_{(2,1^{n-2})}^I&=\sum_{k\in I} \binom{n}{k}(k-1)! \tilde{\chi}_{(2,1^{n-2})}((k,1^{n-k})) \notag \\
                           &=\sum_{k\in I} \binom{n}{k}(k-1)! (-1)^{k-1}\tilde{\chi}_{(n-1,1)}((k,1^{n-k})) \notag \\
                           &=\sum_{k\in I} \binom{n}{k}(k-1)! (-1)^{k-1} \frac{n-k-1}{n-1}, \notag \\
                           &\leq \lambda_{(n-1,1)}^I. \notag 
\end{align} 
We conjecture that the second largest eigenvalue of $\mathrm{Cay}(S_n,C(n,I))$ in part (c) of Theorem \ref{thm:n-1} can only be attained by $(n-1,1)$, $(2^2,1^{n-4})$ or $(2,1^{n-2})$:

\begin{cx}
Suppose $n\ge 7$ is odd and $\{n-1\}\subseteq I\subseteq \{2,3,\ldots,n-1\}$. Then 
$$
\lambda_{(n-3,2,1)}^I<\max\left\{\lambda_{(n-1,1)}^I, \lambda_{(2^2, 1^{n-4})}^I\right\}.
$$ 
\end{cx}

\section{$\mathrm{Cay}(S_n,C(n,I))$ with $I\cap\{n-1,n\}=\{n\}$} \label{sec:n}

\begin{lem}
\label{cor:lower_bound_1^n_n-1,1_odd}
Suppose $n\ge 7$ is odd and $\{n\}\subseteq I \subseteq\{2,3,\ldots,n-2,n\}$. Then 
$$
\lambda_{(1^n)}^I-\lambda_{(n-1,1)}^I>\frac{n}{2}(n-2)!.
$$ 
\end{lem}

\begin{proof}
Similarly to the proof of Lemma~\ref{lem:diff_(1^n)_(n-1,1)_odd}, we have
  \begin{eqnarray} 
  \lambda_{(1^n)}^I-\lambda_{(n-1,1)}^I&=&n(n-2)!+\sum_{k\in I\setminus\{n\}}\binom{n}{k}(k-1)!\left((-1)^{k-1}-\frac{n-k-1}{n-1}\right)\nonumber\\
   &\ge& n(n-2)!-\sum\limits_{\substack{2\leq k\leq n-3\\ k\text{~is~even}}}\binom{n}{k}(k-1)!\frac{2n-k-2}{n-1}\nonumber\\
   &=& n(n-2)!\left(1-\sum\limits_{\substack{2\leq k\leq n-3\\ k\text{~is~even}}} \frac{2n-k-2}{k(n-k)!}\right)\nonumber \\
   &=& n(n-2)!\left(\frac{1}{(n-5)!}-\sum\limits_{\substack{2\leq k\leq 5\\ k\text{~is~even}}} \frac{2n-k-2}{k(n-k)!}\right) \nonumber\\
   & & + n(n-2)!\left(\sum\limits_{\substack{6\leq k\leq n-3\\ k\text{~is~even}}}\frac{1}{(n-k-1)!}-\frac{1}{(n-k+1)!}-\frac{2n-k-2}{k(n-k)!}\right)  \nonumber \\
   & & +n(n-2)!\left(1-\frac{1}{2}\right) \nonumber \\
   &>& \frac{n}{2}(n-2)!, \nonumber
\end{eqnarray}
where the last step follows from the fact that the right-hand side of (\ref{ineq:diff_1^n_n-1,1_odd}) is positive when taking $k_0=n-2$.
\end{proof}

\begin{lem}\label{cor:n_even_1^n_n-1,1}
Suppose $n\ge 8$ is even and $\{n\}\subseteq I\subseteq\{2,3,\ldots,n-2,n\}$. Then $\lambda_{(n-1,1)}^I>\lambda_{(1^n)}^I$.
\end{lem}

\begin{proof}
Similarly to the proof of Lemma~\ref{lem:diff_(1^n)_(n-1,1)_even}, we have
  \begin{align}
  & \lambda_{(n-1,1)}^I-\lambda_{(1^n)}^I \nonumber \\
  &\ge   n(n-2)!\left(\frac{n-2}{n}-\sum\limits_{\substack{2\leq k\leq n-3\\k~\text{is~odd}}}\frac{1}{(n-k)!} \right) \nonumber \\
  &= n(n-2)!\left(\frac{2n-6}{4(n-4)!}-\sum\limits_{\substack{2\leq k\leq 4\\k~\text{is~odd}}}\frac{1}{(n-k)!} \right) \nonumber\\
  &  + n(n-2)!\left( \sum\limits_{\substack{5\leq k\leq n-3 \\k~\text{is~odd} }} \frac{2n-k-3}{(k+1)(n-k-1)!}-\frac{2n-k-1}{(k-1)(n-k+1)!}-\frac{1}{(n-k)!}\right) \nonumber \\
  & + n(n-2)!\left(\frac{n-2}{n}-\frac{n}{2(n-2)}\right) \nonumber \\
  &> 0. \nonumber \qedhere
 \end{align}
\end{proof}

\begin{lem}\label{lem:n-2,1,1}
Suppose $n\ge 9$ and $\{n\}\subseteq I \subseteq \{2,3,\ldots,n-2, n\}$.  Then for any $\zeta\ne (2^2,1^{n-4})$, $(n-m,1^m)$ with $0\leq m\leq n-1$, we have $\lambda_\zeta^I<\lambda_{(n-2,1^2)}^I$. 
\end{lem}

\begin{proof}
Suppose $\zeta$ is any partition of $n$ other than $(2^2,1^{n-4})$ and $ (n-m,1^m)$ with $0\leq m\leq n-1$. By Lemma~\ref{lem:upper_bound}, we have $\tilde{\chi}_{\zeta}((k,1^{n-k}))<\frac{(n-k)(n-k-1)}{n(n-1)}$ for $2\leq k\leq n-3$. According to Table~\ref{tab:tab4}, we have $\tilde{\chi}_{\zeta}((n-2,1^2))\leq \frac{6}{n(n-1)(n-5)}$. Define 
$$
f(n,k)=\begin{cases}
    	\frac{(n-k)(n-k-1)}{n(n-1)}, & \text{if}~2\leq k\leq n-3;\\ 
    	\frac{6}{n(n-1)(n-5)}, & \text{if}~k=n-2.    
\end{cases}
$$ 
Note that $\tilde{\chi}_{\zeta}((n))=0$ by Lemma~\ref{lem:character_for_n_cycles}.
Therefore,
\begin{equation}\label{ineq:n-2,1,1_1}
\lambda_\zeta^I = \sum_{k\in I\setminus\{n\}} \binom{n}{k}(k-1)! \tilde{\chi}_\zeta((k,1^{n-k})) \leq \sum_{k\in I\setminus\{n\}} \binom{n}{k}(k-1)! f(n,k). \nonumber
\end{equation}
One can further verify that $\tilde{\chi}_{(n-2,1^2)}((k,1^{n-k}))< f(n,k)$ for $2\leq k\leq n-2$. Hence
	\begin{eqnarray*} 
		& & \lambda_{(n-2,1^2)}^I-\lambda_{\zeta}^I \nonumber \\
		&\ge& 2(n-3)!+ \sum_{k\in I\setminus\{n\}} \binom{n}{k}(k-1)! \left(\frac{(n-k-1)(n-k-2)-2c_2}{(n-1)(n-2)}-f(n,k) \right)\nonumber \\
    &\ge& 2(n-3)!+ \sum_{k=2}^{n-2} \binom{n}{k}(k-1)! \left(\frac{(n-k-1)(n-k-2)-2c_2}{(n-1)(n-2)}-f(n,k) \right)\nonumber \\
		&= & 2(n-3)!+ \sum_{k=2}^{n-3} \binom{n}{k}(k-1)! \left(\frac{(n-k-1)(n-k-2)-2c_2}{(n-1)(n-2)}- \frac{(n-k)(n-k-1)}{n(n-1)}\right)  \nonumber \\
		& & - \frac{3(n-3)!}{n-5} \nonumber \\
		&=& 2(n-3)!- \sum_{k=2}^{n-3} \frac{2(n-3)!}{(n-k)(n-k-2)!} -\frac{n}{n-2}- \frac{3(n-3)!}{n-5} \nonumber \\
		&=& 2(n-3)!\left(1-\sum_{k=2}^{n-3} \frac{1}{(n-k)(n-k-2)!} \right) -\frac{n}{n-2}- \frac{3(n-3)!}{n-5}\nonumber \\ 
		&=& 2(n-3)!\left(\frac{1}{2}+\frac{1}{(n-2)!}\right) -\frac{n}{n-2}- \frac{3(n-3)!}{n-5} \nonumber \\
		&>& 0.
	\end{eqnarray*}
This completes the proof.
\end{proof}

The following is the main result in this section.

\begin{thm}\label{thm:n}
Suppose $n\ge 7$ and $\{n\}\subseteq I \subseteq \{2,3,\ldots,n-2,n\}$. Then the following statements hold:
\begin{itemize}
\item[\rm (a)] if $n$ is odd and $I$ contains at least one even number, then the second largest eigenvalue of $\mathrm{Cay}(S_n,C(n,I))$ is attained uniquely by $(1^n)$, and moreover the multiplicity of this eigenvalue is $1$;
\item[\rm (b)] if $n$ is odd and $I$ only contains odd numbers, then the strictly second largest eigenvalue of $\mathrm{Cay}(S_n,C(n,I))$ is attained by $(n-2,1^2)$ and $(3,1^{n-3})$, and moreover the multiplicity of this eigenvalue is $\frac{(n-1)^2(n-2)^2}{2}$;
\item[\rm (c)] if $n$ is even, then the second largest eigenvalue of $\mathrm{Cay}(S_n,C(n,I))$ is attained by $(n-2,1^2)$ or $(2,1^{n-2})$. 
\end{itemize}
\end{thm}

\begin{proof}
(a) For any $\zeta\vdash n$, we have 
	\begin{eqnarray}
		\lambda_\zeta^I= \sum_{k\in I\setminus\{n\}} \binom{n}{k}(k-1)! \tilde{\chi}_\zeta((k,1^{n-k}))+ (n-1)! \tilde{\chi}_\zeta((n)). \nonumber
	\end{eqnarray}
According to Lemma~\ref{lem:character_for_n_cycles}, if $\zeta\ne (n-m,1^m)$ with $0\leq m\leq n-1$, then $\tilde{\chi}_\zeta((n))=0$. Thus, by Lemma~\ref{lem:cycle_n-2}, for any $\zeta\ne (n-m,1^m)$ with $0\leq m\leq n-1$,
	\begin{eqnarray}
		\lambda_\zeta^I &=& \sum_{k\in I\setminus\{n\}} \binom{n}{k}(k-1)! \tilde{\chi}_\zeta((k,1^{n-k})) \nonumber\\
		                &<& \sum_{k\in I\setminus\{n\}} \binom{n}{k}(k-1)! \tilde{\chi}_{(n-1,1)}((k,1^{n-k})) \nonumber \\
		                &=& \lambda_{(n-1,1)}^I-(n-1)! \tilde{\chi}_{(n-1,1)}((n)) \nonumber\\
		                &=& \lambda_{(n-1,1)}^I+(n-2)!. \label{eq:21}
	\end{eqnarray}
   
Now suppose $\zeta=(n-m,1^m)$ with $2\leq m\leq n-2$. Since $n\ge 7$ is odd, $\{n\}\subset I \subseteq \{2,3,\ldots,n-2,n\}$ and $I$ contains at least one even number less than $n-2$, by Lemmas~\ref{lem:character_for_n_cycles} and \ref{lem:cycle_n-2} we obtain 
    \begin{eqnarray}
    	\lambda_\zeta^I&=&\sum_{k\in I\setminus\{n\}} \binom{n}{k}(k-1)! \tilde{\chi}_\zeta((k,1^{n-k})) + (n-1)! \frac{\chi_\zeta((n))}{\chi_\zeta(\mathbf{1})}\nonumber \\
    	&=& \sum_{k\in I\setminus\{n\}} \binom{n}{k}(k-1)! \tilde{\chi}_\zeta((k,1^{n-k})) + (-1)^m (n-m-1)!m! \nonumber \\
    	&<& \sum_{k\in I\setminus\{n\}} \binom{n}{k}(k-1)! \tilde{\chi}_{(n-1,1)}((k,1^{n-k})) + 2(n-3)! \nonumber \\
    	&=& \lambda_{(n-1,1)}^I +(n-2)!+2 (n-3)! \nonumber \\
    	&=& \lambda_{(n-1,1)}^I+n(n-3)!  \nonumber \\
    	&<& \lambda_{(n-1,1)}^I+2(n-2)! .\label{eq:22}
    \end{eqnarray}
On the other hand, by Lemma~\ref{cor:lower_bound_1^n_n-1,1_odd} we have 
\begin{eqnarray}\label{eq:23}
	\lambda_{(1^n)}^I-\lambda_{(n-1,1)}^I>\frac{n}{2}(n-2)! .
\end{eqnarray}
It follows from (\ref{eq:21}), (\ref{eq:22}) and (\ref{eq:23}) that $(1^n)$ is the unique partition of $n$ whose corresponding Specht module achieves the second largest eigenvalue of $\mathrm{Cay}(S_n,C(n,I))$. By \eqref{eq:multi}, the multiplicity of this eigenvalue is equal to the square of the degree of the sign representation $\rho_{(1^n)}$, which is equal to $1$.  
 
(b) One can easily verify the result for $n=7$. Now suppose $n\ge 9$. As there are only odd numbers in $I$, $\mathrm{Cay}(S_n,C(n,I))$ has exactly two connected components and $\lambda_\zeta^I=\lambda_{\zeta'}^I$ for any $\zeta\vdash n$. Thus the largest eigenvalue of $\mathrm{Cay}(S_n,C(n,I))$ is attained by $(n)$ and $(1^n)$. Since $n$ is odd, one can verify that the function $f(n,k)$ defined in the proof of Lemma~\ref{lem:n-2,1,1} also satisfies $f(n,k)\ge \tilde{\chi}_{(2^2,1^{n-4})}((k,1^{n-k}))$ for $2\leq k\leq n-2$. Thus the result in Lemma~\ref{lem:n-2,1,1} actually applies to any $\zeta\ne (n-m,1^m)$ with $0\leq m\leq n-1$, that is, $\lambda_{\zeta}^I<\lambda^I_{(n-2,1^2)}$. Since $I$ contains only odd numbers, we have $I\ne \{2,3,\ldots,n-2,n\}$. So by parts (a) and (b) of Lemma~\ref{re:n-m,1^m}, we get $\lambda^I_{(n-2,1^2)}>\lambda^I_{(n-1,1)}=\lambda^I_{(2,1^{n-2})}$ and the maximum of $\lambda_{(n-m,1^m)}^I$ for $1\leq m\leq n-2$ can only be attained by $m=2,n-3$. That is, the strictly second largest eigenvalue of $\mathrm{Cay}(S_n,C(n,I))$ is attained by $(n-2,1^2)$ and $(3,1^{n-3})$. As the dimensions of $\rho_{(n-2,1^2)}$ and $\rho_{(3,1^{n-3})}$ are both $\frac{(n-1)(n-2)}{2}$, the multiplicity of this eigenvalue is $\frac{(n-1)^2(n-2)^2}{2}$ by equation \eqref{eq:multi}.
 
(c) One can verify that the result is true for $n=8$. Now suppose $n\ge 10$. First, by Lemma~\ref{lem:n-2,1,1}, for any $\zeta\ne (2^2,1^{n-4})$, $(n-m,1^m)$ with $0\leq m\leq n-1$, we have $\lambda_{\zeta}^I<\lambda_{(n-2,1^2)}^I$. Second, by parts (a) and (b) of Lemma~\ref{re:n-m,1^m}, the maximum of $\lambda_{(n-m,1^m)}^I$ for $1\leq m\leq n-1$ can only be attained by $m=1,2,n-3,n-2$ or $n-1$, and $\lambda_{(n-2,1^2)}^I>\lambda_{(3,1^{n-3})}^I$ as $I$ contains the even number $n$. Thirdly, by Lemma~\ref{cor:n_even_1^n_n-1,1} and parts (b) and (c) of Lemma~\ref{re:n-m,1^m}, we have $\lambda_{(1^n)}^I<\lambda_{(n-1,1)}^I<\max\left\{\lambda_{(n-2,1^2)}^I,\lambda_{(2,1^{n-2})}^I\right\}$. Therefore, the second largest eigenvalue of $\mathrm{Cay}(S_n,C(n,I))$ can only be attained by $(n-2,1^2),~(2,1^{n-2})$ and $(2^2,1^{n-4})$. 
   % First,
%\begin{eqnarray}
%	& &\lambda_{(n-2,1,1)}^I-\lambda_{(1^n)}^I\nonumber \\
%	&=& \sum_{k\in I\setminus \{n\}} \binom{n}{k}(k-1)!\left(\frac{(n-k-1)(n-k-2)-2c_2}{(n-1)(n-2)}-(-1)^{k-1}\right) + 2(n-3)!+(n-1)! \nonumber \\
%	&\ge & \sum_{\substack{2\leq k\leq n-2\\ k~\mathrm{is~odd}}} \binom{n}{k}(k-1)! \left\frac{(n-k-1)(n-k-2)}{(n-1)(n-2)}-1\right) + 2(n-3)!+(n-1)! \nonumber \\
%	&>& -\sum_{\substack{2\leq k\leq n-2\\ k~\mathrm{is~odd}}} \binom{n}{k}(k-1)! + 2(n-3)!+(n-1)! \nonumber\\
%	&>& 0.
%\end{eqnarray}
%The last step is because 
%\begin{eqnarray}
%	& &\sum_{\substack{2\leq k\leq n-2\\ k~\mathrm{is~odd}}} \binom{n}{k}(k-1)! \nonumber\\
%	&=& \sum_{\substack{2\leq k\leq n-2\\ k~\mathrm{is~odd}}} \frac{n!}{k(n-k)!} \nonumber \\
%	&=&  \sum_{\substack{2\leq k\leq n-2\\ k~\mathrm{is~odd}}} \frac{1}{(n-k)!} \left((n-1)!+(n-k)(n-2)!+(n-k)(n-1-k)\frac{(n-2)!}{k}\right) \nonumber \\
%	&<& \frac{1}{3}(n-1)!+\frac{7}{12} (n-2)!+\frac{4}{3}(n-2)! \nonumber \\
%	&<& \frac{1}{3} (n-1)!+2(n-2)!. \nonumber  
%\end{eqnarray}
We now show that we can rule out $(2^2,1^{n-4})$. In fact, since $n$ is even, we have $\tilde{\chi}_{(n-2,1^2)}((k,1^{n-k}))>\tilde{\chi}_{(2^2,1^{n-4})}((k,1^{n-k}))$ for every $2\leq k\leq n$ except $k=n-2$. Thus,
\begin{eqnarray}
	& &\lambda_{(n-2,1^2)}^I-\lambda_{(2^2,1^{n-4})}^I\nonumber \\
	&=&\sum_{k\in I} \binom{n}{k}(k-1)!\left(\frac{(n-k-1)(n-k-2)-2c_2}{(n-1)(n-2)}+(-1)^k\frac{(n-k)(n-k-3)+2c_2}{n(n-3)}  \right) \nonumber\\ 
	&\ge& \sum_{k\in \{n-2,n\}} \binom{n}{k}(k-1)!\left(\tilde{\chi}_{(n-2,1^2)}((k,1^{n-k}))-\tilde{\chi}_{(2^2,1^{n-4})}((k,1^{n-k}))\right) \nonumber \\
	&=& 2(n-3)!-(n-1)(n-4)! \nonumber\\
	&>&0,  \label{ineq:36}
\end{eqnarray} 
from which the desired result follows. 
\end{proof}

The following is an immediate corollary of Theorem \ref{thm:n}.

\begin{cor}\label{cor:aldous_n}
	Suppose $n\ge 7$ and $\{n\}\subseteq I \subseteq \{2,3,\ldots,n-2,n\}$. Then $\mathrm{Cay}(S_n,C(n,I))$ does not possess the Aldous property.
\end{cor}

\section{$\mathrm{Cay}(S_n,C(n,I))$ with $\{n-1,n\}\subseteq I$} % (fold)
\label{sec:n1n}

\begin{lem}\label{cor:odd_n_n-1}
Suppose $n\ge 7$ is odd and $\{n-1,n\}\subseteq I \subseteq\{2,3,\ldots,n-1,n\}$. Then the following statements hold:
  \begin{itemize}
    \item[\rm (a)] if the largest number in $I\setminus \{n-1,n\}$ is odd and $I\ne\{2,3,n-1,n\}$, then $\lambda_{(1^n)}^I>\lambda_{(n-1,1)}^I$;
    \item[\rm (b)] if the largest number in $I\setminus \{n-1,n\}$  is even, then $\lambda_{(n-1,1)}^I>\lambda_{(1^n)}^I$;
    \item[\rm (c)] if $I=\{2,3,n-1,n\}$ or $\{n-1,n\}$, then $\lambda_{(n-1,1)}^I=\lambda_{(1^n)}^I$. 
  \end{itemize}
\end{lem}

\begin{proof}
Note that
\begin{eqnarray*} 
  \lambda_{(1^n)}^I-\lambda_{(n-1,1)}^I
  &=& \sum_{k\in I} \binom{n}{k}(k-1)!\left((-1)^{k-1}-\frac{n-k-1}{n-1}\right) \nonumber \\
  &=& \sum_{k\in I\setminus \{n-1,n\}} \binom{n}{k}(k-1)!\left((-1)^{k-1}-\frac{n-k-1}{n-1}\right) \nonumber \\
  &=& \lambda_{(1^n)}^{I\setminus \{n-1,n\}}-\lambda_{(n-1,1)}^{I\setminus \{n-1,n\}}.
\end{eqnarray*}
Thus, if $I=\{n-1,n\}$, then $\lambda_{(1^n)}^{I}=\lambda_{(n-1,1)}^I$, and if $I\ne \{n-1,n\}$, then we obtain the desired results by applying Lemmas~\ref{lem:diff_(1^n)_(n-1,1)_odd} and \ref{lem:diff_(1^n)_(n-1,1)_even} directly to $I \setminus \{n-1,n\} \neq \emptyset$.  
\end{proof}

% section $\boldsymbol{\mathrm{cay_s_n_c_n_i_cay_with_$\boldsymbol{i\cap\{n-1,n\_{n_1_n}_n_1_n_ (end)

\begin{lem}\label{lem:last_one}
Suppose $n\ge 7$ and $\{n,n-1\}\subseteq I \subseteq \{2,3,\ldots,n\}$. If $n$ is even, then $\lambda_{(n-2,1^2)}^I>\lambda_{(n-m,2,1^{m-2})}^I$ for $m=2,3,n-3,n-2$; if $n$ is odd, then $\lambda_{(n-2,1^2)}^I>\lambda_{(n-m,2,1^{m-2})}^I$ for $m=2,3,n-3$.
\end{lem}

\begin{proof}
With the help of Table~\ref{tab:tab1} one can verify that for $m=3, n-3$ we have 
$$
 \sum_{k=n-1}^n\binom{n}{k}(k-1)!\tilde{\chi}_{(n-2,1^2)}((k,1^{n-k}))>\sum_{k=n-1}^n\binom{n}{k}(k-1)!\tilde{\chi}_{(n-m,2,1^{m-2})}((k,1^{n-k}))
$$	
and 
$$
 \tilde{\chi}_{(n-2,1^2)}((k,1^{n-k}))\ge \tilde{\chi}_{(n-m,2,1^{m-2})}((k,1^{n-k})) \quad \text{for~any~}2\leq k\leq n-2.
$$ 
Since $\{n-1,n\}\subseteq I \subseteq \{2,3,\ldots,n\}$, it follows that $\lambda_{(n-2,1^2)}^I>\lambda_{(n-m,2,1^{m-2})}^I$ for $m=3, n-3$.
 
We have $\tilde{\chi}_{(n-2,1^2)}((k,1^{n-k}))>\tilde{\chi}_{(n-2,2)}((k,1^{n-k}))$ for $3\leq k\leq n$. We also have $\sum_{k\in\{2,n\}}\binom{n}{k}(k-1)!\tilde{\chi}_{(n-2,1^2)}((k,1^{n-k}))>\sum_{k\in \{2,n\}}\binom{n}{k}(k-1)!\tilde{\chi}_{(n-2,2)}((k,1^{n-k}))$. Thus $\lambda_{(n-2,1^2)}^I>\lambda_{(n-2,2)}^I$ whenever $n \in I\subseteq \{2,3,\ldots,n\}$. Inequality~(\ref{ineq:36}) implies that $\lambda_{(n-2,1^2)}^I>\lambda_{(2^2,1^{n-4})}^I$ whenever $n$ is even and $n \in I\subseteq \{2,3,\ldots,n\}$.
\end{proof}

The main result in this section is as follows.

\begin{thm}\label{thm:n_n-1}
Suppose $n\ge 7$ and $\{n,n-1\}\subseteq I \subseteq \{2,3,\ldots,n\}$. Then the following statements hold:
\begin{itemize}
\item[\rm (a)] if $n$ is even, then the second largest eigenvalue of $\mathrm{Cay}(S_n, C(n,I))$ can only be attained by $(1^n),(n-2,1^2)$ or $(2,1^{n-2})$;
\item[\rm (b)] if $n$ is odd, then the second largest eigenvalue of $\mathrm{Cay}(S_n, C(n,I))$ can only be achieved by $(1^n),(n-2,1^2),(3,1^{n-3})$ or $(2^2,1^{n-4})$.
\end{itemize}
\end{thm}

\begin{proof} 
(a) One can verify that the result is true for $n=8$. Now suppose $n\ge 10$ and $n$ is even.
According to Lemmas~\ref{lem:character_for_n-1_cycles} and \ref{lem:n-2,1,1}, if $\zeta\ne (n-m,1^m), (n-m,2,1^{m-2})$, then $\lambda_{\zeta}^I=\lambda_\zeta^{I\setminus\{n-1\}}<\lambda_{(n-2,1^2)}^{I\setminus \{n-1\}}=\lambda_{(n-2,1^2)}^I$. On the other hand, by Lemma~\ref{re:n-m,1^m} we have $\lambda_{(n-1,1)}^I<\max\left\{\lambda_{(n-2,1^2)}^I,\lambda_{(2,1^{n-2})}^I\right\}$ and $\lambda_\zeta^I<\lambda_{(n-2,1^2)}^I$ for $\zeta=(n-m,1^m)$ with $3\leq m\leq n-3$. Hence $\lambda_{\zeta}^I<\max\left\{\lambda_{(n-2,1^2)}^I,\lambda_{(2,1^{n-2})}^I\right\}$ for any $\zeta \ne (n),(1^n),(n-2,1^2),(2,1^{n-2})$, $(n-m,2,1^{m-2})$ with $2\leq m\leq n-2$.

Since $n-1\in I\subseteq \{2,3,\ldots,n-1,n\}$, by Remark~\ref{re:(n-m,2,1^m-2)} the maximum of $\lambda^I_{(n-m,2,1^{m-2})}$ for $2\leq m\leq n-2$ can only be attained by $m=2,3,n-3$ or $n-2$. Moreover, Lemma~\ref{lem:last_one} implies that $\lambda_{(n-m,2,1^{m-2})}^I<\lambda_{(n-2,1^2)}^I$ for $m=2,3,n-3,n-2$. Therefore, the second largest eigenvalue of $\mathrm{Cay}(S_n, C(n,I))$ can only be attained by $(1^n),(n-2,1^2)$ or $(2,1^{n-2})$. 
 
(b) One can easily verify this result for $n=7$. Now suppose $n\ge 9$ and $n$ is odd. Similarly to the proof of part (a) above, one can prove that $\lambda_{\zeta}^I=\lambda_\zeta^{I\setminus\{n-1\}}<\lambda_{(n-2,1^2)}^{I\setminus \{n-1\}}=\lambda_{(n-2,1^2)}^I$ for any $\zeta$ other than $(n-m,1^m)$ and $(n-m,2,1^{m-2})$. By Lemma~\ref{re:n-m,1^m}, we have $\lambda_{(n-1,1)}^I\leq \lambda_{(n-2,1^2)}^I$ and $\lambda_{(n-m,1^m)}^I < \lambda_{(n-2,1^2)}^I$ for $3\leq m\leq n-4$. Note that, if $\lambda_{(n-1,1)}^I= \lambda_{(n-2,1^2)}^I$, then $I = \{2,3,\ldots,n\}$ and thus $\lambda_{(1^n)}^I>\lambda_{(n-1,1)}^I$ by Lemma~\ref{cor:odd_n_n-1}. A straightforward computation shows that $\lambda_{(2,1^{n-2})}^I\leq \lambda_{(n-1,1)}^I<\max\left\{\lambda_{(n-2,1^2)}^I,\lambda_{(1^n)}^I\right\}$. Finally, by Lemma \ref{lem:last_one} we have $\lambda_{(n-2,1^2)}^I>\lambda_{(n-m,2,1^{m-2})}^I$ for $m=2,3,n-3$. Combining all these with Remark~\ref{re:(n-m,2,1^m-2)}, we obtain the desired result.	
\end{proof}

Theorem \ref{thm:n_n-1} implies the following result.

\begin{cor}\label{cor:aldous_n-1_n}
Suppose $n\ge 7$ and $\{n,n-1\}\subseteq I \subseteq \{2,3,\ldots,n\}$. Then $\mathrm{Cay}(S_n,C(n,I))$ does not have the Aldous property.
\end{cor}

We conjecture that the second largest eigenvalue of $\mathrm{Cay}(S_n,C(n,I))$ in part (b) of Theorem \ref{thm:n_n-1} can only be achieved by $(1^n),(n-2,1^2)$ or $(2^2,1^{n-4})$:  

\begin{cx}
Suppose $n\ge 7$ is odd and $\{n-1,n\}\subseteq I\subseteq \{2,3,\ldots,n\}$. Then 
$$
\lambda_{(3,1^{n-3})}^I<\max\left\{\lambda_{(1^n)}^I, \lambda_{(n-2,1^2)}^I, \lambda_{(2^2, 1^{n-4})}^I\right\}.
$$ 
\end{cx}

Note that, by Lemma~\ref{re:n-m,1^m}, we already know that $\lambda_{(3,1^{n-3})}^I \le \lambda_{(n-2,1^2)}^I$ and the equality holds if and only if $I$ contains only odd numbers other than $n-1$ and $n-2$.

%\begin{table}[h]
%\newcolumntype{g}{>{\columncolor{darkblue}}r} 
%\begin{tabular}{|c|c| c | c|}
%\hline
%\rowcolor{darkgray} \multicolumn{2}{|c|}{Cases}  & partitions & multiplicity  \\\hline
%\multirow{5}*{$I\subseteq \{2,3,\ldots,n-2\}$ } & only even & $(n-1,1)$ & $(n-1)^2$ \\ \cline{2-4}
% & only odd   & $(n-1,1) \& (2,1^{n-2})$ & $2(n-1)^2$\\ \cline{2-4}
% & $I=\{2,3\}$ & $(n-1,1) \& (1^n)$ & $(n-1)^2+1$\\ \cline{2-4}
% & even \& odd, odd largest & $(1^n)$ & 1 \\ \cline{2-4}
% & even \& odd, even largest & $(n-1,1)$ & $(n-1)^2$\\ \hline
%\multirow{3}*{$I\subseteq\{2,3,\ldots,n-1\}$ } & $n$ is even, even \& odd  & $(1^n)$  & $1$\\\cline{2-4}
% & $n$ is even, only odd  & $(n-1,1),(2,1^{n-2})$ or  &  \\  
% &  &$(n-3,2,1),(3,2,1^{n-3})$ & \\ \cline{2-4}
%\multirow{2}*{$n-1\in I$} &  \multirow{2}*{$n$ is odd} & $(2,2,1^{n-4})$ or $(n-3,2,1)$ & \\ 
% & &  or $(n-1,1)$ or $(2,1^{n-2})$ & \\ \hline
%\multirow{2}*{$I\subseteq \{2,3,\ldots,n-2,n\}$}  & $n$ is  odd, even \& odd & $(1^n)$ & 1 \\ \cline{2-4}
% & $n$ is odd, only odd  & $(n-2,1,1) \& (3,1^{n-3})$ & $(n-1)(n-2)$ \\ \cline{2-4}
%$n\in I$ & $n$ is even  & $(2,1^{n-2})$ or $(n-2,1,1)$ & \\ \hline
% \multirow{4}*{$\{n,n-1\} \subseteq I$} & \multirow{2}*{$n$ is even} & $(1^n)$, $(n-2,1,1)$ & \\
%& & or $(2,1^{n-2})$ & \\ \cline{2-4}
% & \multirow{2}*{$n$ is odd} & $(1^n)$, $(n-2,1,1)$  & \\ 
% & &  or $(2,2,1^{n-4})$ & \\ \hline
%\end{tabular}
%\end{table}

\bigskip
\noindent \textbf{Acknowledgement} 
\medskip

The first author was supported by the Melbourne Research Scholarship provided by The University of Melbourne.


\begin{thebibliography}{}


%\bibitem {A1} D. Aldous. https://www.stat.berkeley.edu/users/aldous/Research/OP/sgap.html.

\bibitem{A} N. Alon. Eigenvalues and expanders. Combinatorica, 6(2):83--96, 1986.

\bibitem{AM} N. Alon and V. D. Milman. $\lambda_1$, isoperimetric inequalities for graphs, and superconcentrators. J. Combin. Theory Ser. B, 38(1):73--88, 1985. 

\bibitem{CLR} P. Caputo, T. M. Liggett and T. Richthammer. Proof of Aldous' spectral gap conjecture.
J. Amer. Math. Soc., 23(3):831--851, 2010.

%\bibitem{BH} A. E. Brouwer and W. H. Haemers. Spectra of Graphs. Springer, 2011.
\bibitem{C1} F. Cesi. Cayley graphs on the symmetric group generated by initial reversals have unit spectral gap. Electron. J. Combin., 16(1):N29, 2009.

%\bibitem{C} F. Cesi. On the eigenvalues of Cayley graphs on the symmetric group generated by a complete multipartite set of transpositions. J. Algebraic Combin., 32(2):155--185, 2010.


%7
%\bibitem{C2} F. Cesi. A few remarks on the octopus inequality and Aldous' spectral gap conjecture. Communications in Algebra, 44(1):279--302, 2016.
%7
%\bibitem{C3} F. Cesi. On the spectral gap of some Cayley graphs on the Weyl group $W(B_n)$. Linear Algebra Appl., 586:274--295, 2020.
%9
%\bibitem{CRS} D. M. Cvetkovi$\acute{c}$, P. Rowlinson, and S. K. Simi$\acute{c}$. An Introduction to the Theory of Graph Spectra. Cambridge Univ. Press, Cambridge, 2010.
%10
\bibitem{CT} F. Chung and J. Tobin. The spectral gap of graphs arising from substring reversals.
Electron. J. Combin., 24(3):P3.4, 2017.  


%\bibitem{DZ} Y.-P. Deng and X.-D. Zhang. A note on eigenvalues of the derangement graph. Ars
%Combin., 101:289--299, 2011.

%13
\bibitem{PS} P. Diaconis and M. Shahshahani. Generating a random permutation with random transpositions. Z. Wahrscheinlichkeitstheor. Verw. Geb., 57(2):159--179, 1981.


\bibitem{D} J.  Dodziuk.  Difference  equations,  isoperimetric  inequality  and  transience  of  certain random walks. Trans. Amer. Math. Soc., 284(2):787--794, 1984. 
%12
%\bibitem{D1} A. B. Dieker. Interlacings for random walks on weighted graphs and the interchange process. SIAM J. Discrete Math. 24(1):191--206, 2010.

%14
%\bibitem{FOW} L. Flatto, A. M. Odlyzko, and D. B. Wales. Random shuffles and group representations.
%Ann. Probab., 13(1):154--178, 1985.

\bibitem{F} W. Fulton. Young Tableaux: With Applications to Representation Theory and Geometry. No. 35. Cambridge Univ. Press, 1997.
%\bibitem{F} J. Friedman. On Cayley graphs on the symmetric group generated by transpositions.
%Combinatorica, 20(4):505-519, 2000.
%15
%\bibitem{GR} C. Godsil and G. Royle. Algebraic Graph Theory, Graduate Texts in Mathematics, vol. 207. Springer, 2001.
%16

%17
%\bibitem{H} F. Harary. Graph Theory, Addison-Wesley (Reading), 1969.

%20
%\bibitem{HJ} S. Handjani and D. Jungreis. Rate of convergence for shuffling cards by transpositions.
%J. Theoret. Probab., 9(4):983--993, 1996.

%21
%\bibitem{HP} J. Hermon and R. Pymar. A direct comparison between the mixing time of the interchange process with ``few" particles and independent random walks, \url{https://arxiv.org/abs/2105.13486}, 2021.
\bibitem{HLW}
S. Hoory, N. Linial and A. Wigderson. Expander graphs and their applications. Bull. Amer. Math. Soc., 43(4):439--561, 2006.

\bibitem{HH} X. Huang and Q. Huang. The second largest eigenvalues of some Cayley graphs on alternating groups. J. Algebraic Combin., 50(1):99--111, 2019.





%\bibitem{HJ1} R. A. Horn and C. R. Johnson. Matrix analysis. Cambridge Univ. Press, second edition, 2012.

%\bibitem{HR} D. Holt and G. Royle. A census of small transitive groups and vertex-transitive graphs. J. Symbolic Comput., 101:51-60, 2020.

%18

\bibitem{HH2} X. Huang, Q. Huang and S. M. Cioab\u{a}. The second eigenvalue of some normal Cayley graphs of highly transitive groups. Electron. J. Combin., 26(2):P2.44, 2019.

%22
%\bibitem{I} I. M. Isaacs. Character Theory of Finite Groups. AMS Chelsea Publishing, Providence, RI, 2006.

%\bibitem{J} G. D. James. The Representation Theory of the Symmetric Groups. Springer Lect. Notes Math. 682, Berlin-Heidelberg-New York, 1978.
%23
\bibitem{JK} G. James and A. Kerber. The Representation Theory of the Symmetric Group. Addison-Wesley, London 1981.

%\bibitem{JL} G. James and M. Liebeck. Representations and Characters of Groups. Cambridge Univ. Press, New York, second edition, 2001.
%24
%\bibitem{K} M. Kassabov. Subspace arrangements and property T. Groups, Geometry, and Dynamics, 5(2):445--477, 2011.


%\bibitem{MS} M. Krebs and A. Shaheen. Expander Families and Cayley Graphs: A Beginner's Guide. Oxford Univ. Press, 2011.

%\bibitem{KLW} C. Y. Ku, T. Lau, and K. B. Wong. The smallest eigenvalues of the 1-point fixing graph. Linear Algebra Appl., 493:433--446, 2016.

%\bibitem{KW} C. Ku and T. Wong. Intersecting families in the alternating group and direct product of symmetric groups. Electron. J. Combin., 14(1): R25, 2007.

%\bibitem{LS} M. Larsen and A. Shalev. Characters of symmetric groups: sharp bounds and applications. Invent. Math., 174(3):645--687, 2008. 

\bibitem{LXZ} Y. Li, B. Xia and S. Zhou. Aldous' spectral gap property for normal Cayley graphs on symmetric groups. European J. Combin., 110:103657, 2023.

\bibitem{LZ}
X. Liu and S. Zhou. Eigenvalues of Cayley graphs. Electron. J. Combin., 29(2):P2.9, 2022.

\bibitem{MR} 
R. Maleki and A. S. Razafimahatratra. On the second eigenvalue of a Cayley graph of the symmetric group. \url{https://arxiv.org/abs/2108.13585}, 2021.

\bibitem{M} 
B. Mohar. Isoperimetric numbers of graphs. J. Combin. Theory Ser. B, 47(3):274--291, 1989.

%\bibitem{N} A. Nilli. On the second eigenvalue of a graph. Discrete Math., 91(2):207--210, 1991.
%27
%\bibitem{O} C. D. Olds. Odd and even derangements. Solution E907, Amer. Math. Monthly, 57:687--688, 1950.

%29
\bibitem{PP} O. Parzanchevski and D. Puder. Aldous's spectral gap conjecture for normal sets. Trans. Amer. Math. Soc., 373(10):7067--7086, 2020.

%28
%\bibitem{PC} D. Piras. Generalizations of Aldous' Spectral Gap Conjecture, Tesi di Laurea, Universit\'a degli Studi Roma Tre, \url{http://www.mat.uniroma3.it/scuola_orientamento/alumni/laureati/piras/sintesi.pdf}, 2010.

%30
%\bibitem{Sachs} H. Sachs. Über selbstkomplementäre graphen[J]. Publ. Math. Debrecen, 9(11):270-288, 1962.

%\bibitem{R} P. Renteln. On the spectrum of the derangement graph. Electron. J. Combin., 14(1):R82, 2007.

\bibitem{Sagan} B. Sagan. The Symmetric Group: Representations, Combinatorial Algorithms, and Symmetric Functions. Vol. 203. Springer Science $\&$ Business Media, 2001.

%\bibitem{S1} J. P. Serre. Linear Representations of Finite Groups. Springer, 1977. 

%33
\bibitem{SZ} J. Siemons and A. Zalesski. On the second largest eigenvalue of some Cayley graphs of the symmetric group.  J. Algebraic Combin., 55(3):989--1005,2022.
%31
%\bibitem{S} B. Steinberg. Representation Theory of Finite Groups: An Introductory Approach. Springer, 2012.
%32
%\bibitem{SLW} H. Shlomo, N. Linial, and A. Wigderson. Expander graphs and their applications, Bull. Amer. Math. Soc., 43(4):439-561, 2006.

%\bibitem{WWX}  C. Q. Wang, D. J. Wang, and M. Y. Xu. On normal Cayley graphs of finite groups. Science in China (Series A), 28(2):131--139, 1998.
%34
\bibitem{Z} P. H. Zieschang. Cayley graphs of finite groups. J. Algebra, 118(2):447--454, 1988.





\end{thebibliography}
\end{document}